\documentclass[dvipsnames,x11names]{article}

\usepackage[utf8]{inputenc}
\usepackage{amsmath}
\usepackage{amsfonts}
\usepackage{amssymb}
\usepackage{amsthm}
\usepackage{graphicx}
\usepackage{subcaption}
\usepackage{epstopdf}
\usepackage[export]{adjustbox}
\usepackage{multirow}
\usepackage{booktabs}
\usepackage{color}
\usepackage[export]{adjustbox}
\usepackage{enumerate}
\usepackage{placeins}
\usepackage[multidot]{grffile}

\usepackage{fp}
\usepackage{pstricks}
\usepackage{tikz, ifthen, rotating}
\usetikzlibrary{decorations.pathreplacing}
\usepackage{pgfplots}

\usepackage{ifdraft}

\DeclareFontFamily{U}{mathx}{\hyphenchar\font45}
\DeclareFontShape{U}{mathx}{m}{n}{
      <5> <6> <7> <8> <9> <10>
      <10.95> <12> <14.4> <17.28> <20.74> <24.88>
      mathx10
      }{}
\DeclareSymbolFont{mathx}{U}{mathx}{m}{n}
\DeclareFontSubstitution{U}{mathx}{m}{n}
\DeclareMathAccent{\widecheck}{0}{mathx}{"71}
\DeclareMathAccent{\wideparen}{0}{mathx}{"75}

\newtheorem{theorem}{Theorem}
\newtheorem{lemma}[theorem]{Lemma}
\newtheorem{remark}[theorem]{Remark}

\theoremstyle{definition}

\newtheorem{example}{Example}

\newcommand{\diam}{\operatorname*{diam}}

\newcommand{\III}[1]{{\textstyle\left\vert\kern-0.25ex\left\vert\kern-0.25ex\left\vert #1 
    \right\vert\kern-0.25ex\right\vert\kern-0.25ex\right\vert}}

\newcommand{\tri}{\mathcal{T}}
\newcommand{\triH}{\tri_H}
\newcommand{\trih}{\tri_h}
\newcommand{\nodesH}{\mathcal{N}_H}
\newcommand{\nodesfreeH}{\mathcal{N}_H^{\rm free}}
\newcommand{\nodesIH}{\mathcal{N}_H^{\rm I}}
\newcommand{\nodesIIH}{\mathcal{N}_H^{\rm II}}

\newcommand{\R}{\mathbb{R}}

\newcommand{\VH}{V_H}

\newcommand{\Vf}{V^{\operatorname*{f}}}

\newcommand{\VmsH}{V^{\operatorname*{ms}}_H}

\newcommand{\VmsHk}{V^{\operatorname*{ms}}_{H,k}}

\newcommand{\vf}{v^{\operatorname*{f}}}
\newcommand{\uf}{u^{\operatorname*{f}}}
\newcommand{\ufk}{u^{\operatorname*{f}}_k}

\newcommand{\umsH}{u^{\operatorname*{ms}}_H}
\newcommand{\ufH}{u^{\operatorname*{f}}}
\newcommand{\vmsHk}{v^{\operatorname*{ms}}_{H,k}}
\newcommand{\umsHk}{u^{\operatorname*{ms}}_{H,k}}
\newcommand{\umsLHk}{u^{\operatorname*{ms,rhs}}_{H,k}}

\newcommand{\IHnodal}{I_H^{\rm nodal}}
\newcommand{\IHsz}{I_H^{\rm SZ}}
\newcommand{\IHproj}{I_H^{A}}
\newcommand{\IHprojs}{I_H^{A,\rm{qm}}}
\newcommand{\IH}{I_H}
\newcommand{\IHinf}{I_{H,1}}
\newcommand{\IIH}{\mathcal{I}_H}
\newcommand{\IHdelta}{I_{H,\delta}}

\newcommand{\Cint}{{\mathcal{I}}}
\newcommand{\id}{\operatorname*{Id}}
\newcommand{\Cdiam}{\ell}
\newcommand{\Cpoinc}{p}

\newcommand{\MW}{M}

\newcommand{\dd}{\, \mathrm{d}}
\newcommand{\fine}{fine}

\DeclareMathOperator*{\supp}{supp}

\newtheorem{assumption}{Assumption}

\usepackage{geometry}
\title{Contrast independent localization of multiscale problems}
\author{Fredrik Hellman%
  \thanks{Department of Information Technology, Uppsala University, Box 337, SE-751 05 Uppsala, Sweden. Supported by Centre for Interdisciplinary Mathematics (CIM), Uppsala University.}
  \and
  Axel M\r{a}lqvist%
  \thanks{Department of Mathematical Sciences,
    Chalmers University of Technology and
    University of Gothenburg 
    SE-412 96 G\"oteborg, Sweden. Supported by the Swedish Research Council.}
}

\begin{document}
\maketitle

\begin{abstract}
  The accuracy of many multiscale methods based on localized
  computations suffers from high contrast coefficients since the
  localization error generally depends on the contrast. We study a
  class of methods based on the variational multiscale method, where
  the range and kernel of a quasi-interpolation operator defines the
  method. We present a novel interpolation operator for two-valued
  coefficients and prove that it yields contrast independent
  localization error under physically justified assumptions on the
  geometry of inclusions and channel structures in the
  coefficient. The idea developed in the paper can be transferred to
  more general operators and our numerical experiments show that the
  contrast independent localization property follows.
\end{abstract}



\section{Introduction}
High contrast and multiscale coefficients are frequently encountered
in partial differential equations (PDEs) for a range of
applications. Typical examples of such coefficients are the
permeability field in porous media flow problems, varying by several
orders of magnitude over short distances, and the rapidly varying heat
conductivity in a composite material. In this paper, we focus on the
multiscale method based on localized orthogonal decomposition (LOD)
\cite{MaPe14} and study how to improve its accuracy for high contrast
coefficients. LOD is based on the framework of the variational
multiscale method (VMS) \cite{HuFeGoMaQu98} in the sense that the full
solution space is decomposed into a coarse and a \fine{} subspace,
where this decomposition is determined by the range and kernel of a
quasi-interpolation operator. A new low-dimensional multiscale space
(subsequently used in a Galerkin or Petrov--Galerkin method) is
constructed by computing coarse basis correctors in the \fine{}
space. The correctors have global support, but can be computed on
localized patches around the support of the coarse basis
functions. The approximability of the multiscale space is determined
by the error introduced by the localization to patches, which in turn
depends on the decay of the correctors within the patch. It was proven
in \cite{MaPe14} that this decay is exponential with respect to the
radius of the patch, independent of the fine-scale variations of the
coefficient but not generally independent of the contrast (ratio
between largest and smallest value) of the coefficient. It has also
been observed in numerical experiments that higher contrast
coefficients lead to slower decay of the correctors, particularly
within connected subdomains with large value of the coefficient,
called channels. The decay of correctors (or the related concept of
fine-scale Green's functions in VMS) in both one- and
multi-dimensional settings was studied in \cite{HuSa07} for different
choices of projection operators and it was clearly shown that the
choice of operator has a large impact on the decay rate.

In this paper, we study coefficients $A$ that take two values $\alpha$
and $1$, where $\alpha \ll 1$. This isolates the effect of high
contrast while still capturing many interesting applications, such as
composite materials and subsurface flows. We introduce a novel
Cl\'ement-type quasi-interpolation operator $\IH$ (based on Scott--Zhang
node variables) whose construction forces corrector decay within
channels and prove that the localization error for this operator is
independent of the contrast. The basic idea is to select the
integration domain for each node variable in such a way the operator
kernel admits a contrast independent Poincar\'e-type inequality within
all channels and inclusions in the domain. In practice, this means
that each connected channel and inclusion needs to have dedicated nodes
placed along its extent with a distance proportional to the mesh size
of the coarse mesh. We present sufficient assumptions on the node
placements and a proof for contrast free localization error when the
decomposition is based on $\IH$. The properties of $\IH$ are studied
both theoretically and numerically. We also present a related operator
$\IHinf$ which numerically performs even better, but for which our
proof does not give any guarantees. They both, however, follow the
basic idea of carefully selecting integration domain for the node
variables, suggesting that this is a key to contrast independent
localization. 

In addition to the VMS based methods, the literature on numerical
homogenization for elliptic multiscale problems includes the
multiscale finite element method (MsFEM) \cite{HoWu97}, the
generalized multiscale finite element method (GMsFEM) \cite{EfGaHo13},
the heterogeneous multiscale method (HMM) \cite{EEn03}, and
polyharmonic homogenization \cite{OwZhBe14}. The issue of high
contrast coefficients without assumptions on periodicity has been
addressed recently by many authors. For example, by using multiscale
finite element approaches in \cite{ChGrHo10, EfGaWu11}, flux norm
approaches in \cite{BeOw2010, OwZh11}, and low-rank approximation of
Green's functions in \cite{Be16}. Our work has similarities to and has
been inspired by the LOD based approach in \cite{PeSc16}, where
contrast independent corrector decay results for $A$-weighted
quasi-interpolation operators were shown under quasi-monotonicity
(\cite{ScVaZi12}) assumptions on the coefficient distribution within
the node patches. The idea of selecting integration domain for the
node variables can be transferred also to the $A$-weighted
quasi-interpolation operators, so that the integration domain is
selected to guarantee quasi-monotonicity. We include numerical
experiments for an $A$-weighted projective operator of that kind
suggesting that carefully selecting integration domain is important
for contrast independent localization.

The outline of the paper is as follows. Section~\ref{sec:problem}
describes the model problem and gives a review of the localized
orthogonal decomposition method for multiscale
problems. Section~\ref{sec:IH} defines the interpolation operator
$\IH$ and investigates its stability and approximability
properties. Section~\ref{sec:error} shows the contrast independent
localization error following from using $\IH$ and presents the total
error of the multiscale method. Section~\ref{sec:num} briefly reviews
a number of additional interpolation operators and presents a series
of numerical experiments where the accuracy for methods based on the
the different operators on high contrast problems is
investigated. Finally, the numerical results are discussed and related
to the theoretical findings.

\section{Problem formulation and numerical method}
\label{sec:problem}
As model problem, we consider the elliptic PDE, with a two-valued coefficient $A$,
\begin{equation}
  \label{eq:poisson}
  -\operatorname*{div} A \nabla u = f
\end{equation}
on a polygonal domain
$\Omega = \Omega^1 \cup \Omega^\alpha \subset \R^d$ ($d = 1, 2$ or
$3$), where $\Omega^1$ and $\Omega^\alpha$ are disjoint. The
coefficient attains two values $A|_{\Omega^1} = 1$ and
$A|_{\Omega^\alpha} = \alpha$ with $0 < \alpha \le 1$ in the two
subdomains $\Omega^1$ and $\Omega^\alpha$. We impose homogeneous
Dirichlet boundary conditions on $\Gamma \subset \partial \Omega$, and
homogeneous Neumann boundary conditions on
$\partial \Omega \setminus \Gamma$. We allow $\Gamma$ to be empty, in
which case compatibility conditions $\int_\Omega u = 0$ and
$\int_\Omega f = 0$ are imposed. We consider the case when $A$ is a
high-contrast and highly oscillatory coefficient, i.e.\ that
$\alpha \ll 1$, and that $\Omega^1$ and $\Omega^\alpha$ cannot be
approximated well as a union of a set of elements, without having the
element diameter be very small. Examples are given in
Figure~\ref{fig:Aexamples}. Note that the coefficient is determined
completely by $\Omega^1$, $\Omega^\alpha$ and $\alpha$.
\begin{figure}[]
  \centering
  \begin{subfigure}[t]{0.4\textwidth}
    \centering
    \includegraphics[width=4cm, frame]{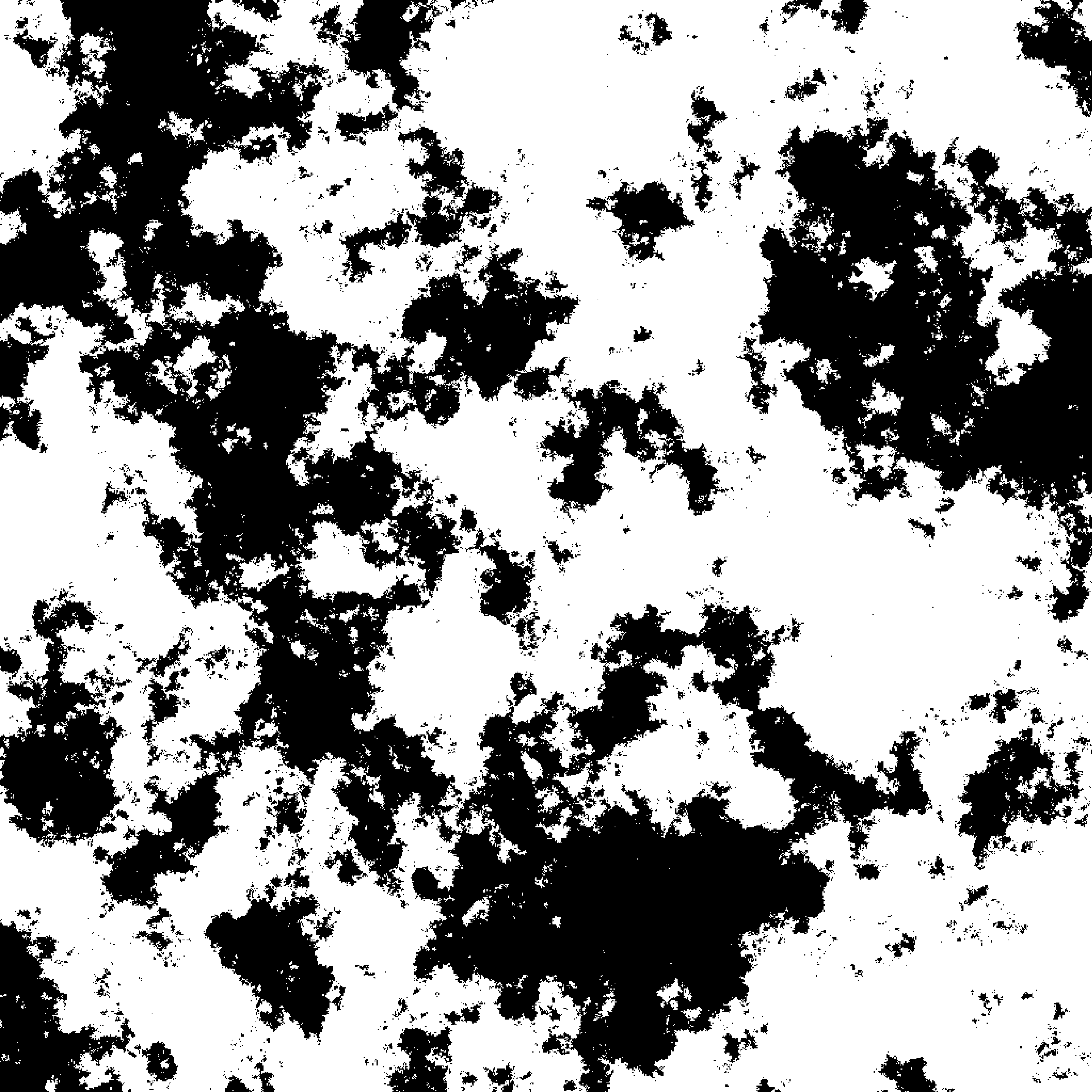}
    \caption{Random field, resembling a heterogeneous porous medium.}
    \label{fig:random_coef}
  \end{subfigure}
  \hspace{1em}
  \begin{subfigure}[t]{0.4\textwidth}
    \centering
    \includegraphics[width=4cm, frame]{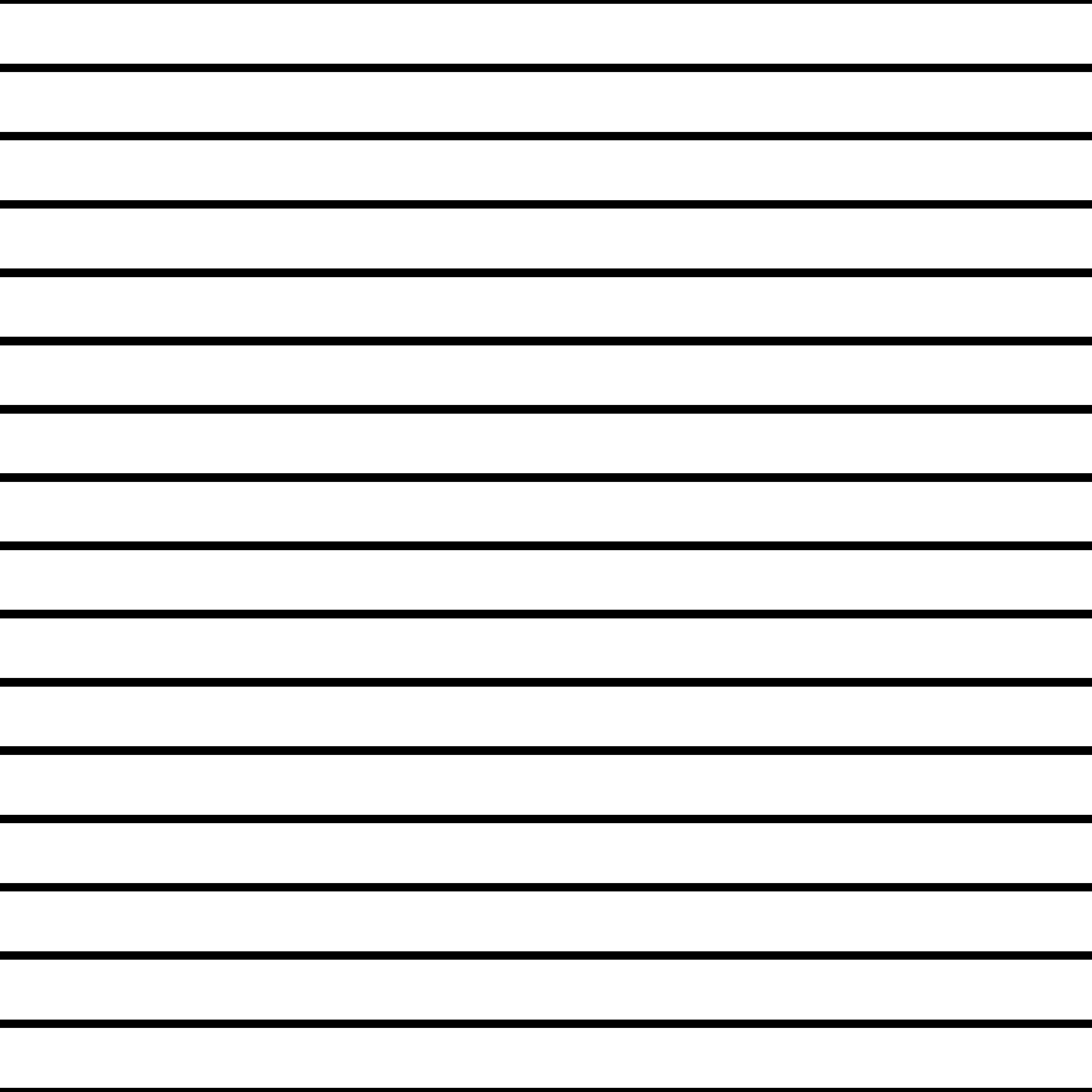}
    \caption{Stripes, resembling cracks or composite materials.}
    \label{fig:stripes_coef}
  \end{subfigure}
  \caption{Examples of coefficient $A$. The square is $\Omega$. Black
    color is $\Omega^1$ and white color is $\Omega^\alpha$.}
  \label{fig:Aexamples}
\end{figure}

We reformulate the problem on weak form. Let
$V = \{v \in H^1(\Omega) \,:\, v|_\Gamma = 0 \}$ and $(\cdot,\cdot)$
denote the scalar product in $L^2(\Omega)$. We introduce a bilinear
form $a$,
\begin{equation*}
  a(u, v) = (A \nabla u, \nabla v),
\end{equation*}
and assume $f \in L^2(\Omega)$. We restate our problem as to find
$u \in V$,
\begin{equation}
  \label{eq:problem}
  a(u, v) = (f, v)
\end{equation}
for all $v \in V$. With the bounds $\alpha \le A \le 1$, we have that
$a$ is bounded and coercive and \eqref{eq:problem} admits a unique
solution by the Lax--Milgram theorem.

We consider the following (coarse scale) finite element
discretization.  Let $\triH$ be a family of conforming triangulations
of $\Omega$ with mesh size parameter $H$. We denote the set of nodes
by $\nodesH$ and the set of free nodes by
$\nodesfreeH \subset \nodesH$. The elements are assumed to be shape
regular and the meshes to be quasi-uniform, i.e.\ we assume there is a
constant $\rho$ independent of $H$ such that for all $H$,
\begin{equation}
  \max_{T\in\triH} \frac{\diam(T)}{\diam(B_T)} \le \rho \quad\text{and}\quad \max_{T,T'\in \triH} \frac{\diam(T)}{\diam(T')} \le \rho,
\end{equation}
where $B_T$ denotes the largest ball contained in $T$. Let $S_H$ be
the standard $\mathcal{P}1$ FE space associated with $\triH$. Further,
let $V_H = S_H \cap V$, satisfying the Dirichlet boundary
conditions. It is well known that convergence for the finite element
method for this problem is generally not achieved unless the
oscillations in $A$ are captured by the mesh. We briefly review the
LOD method which allows for using low-dimensional spaces to find good
approximations to problems with highly oscillatory coefficients. For
more elaborate descriptions of different aspects of this method, see
e.g.\ \cite{EnHeMaPe16, HeMa14, MaPe14}.

\subsection{Quasi-interpolation}
The first step of defining the LOD method is to choose a
quasi-inter\-polation operator $\IIH$ from $V$ onto $\VH$. We call
$V_H$ the coarse space and the infinte-dimensional $V$ the full
space. The coarse space need not resolve the discontinuites of the
coefficient $A$. In practice, the full space is typically also a
finite-dimensional FE space, however, to better convey the new ideas
in this work we use $V$ as full space, since it simplifies the
exposition. See e.g.\ \cite{MaPe14} for the fully discrete
setting. The choice of a quasi-interpolation operator
$\IIH : V \to V_H$ is crucial and defines a \fine{} space $\Vf$ as its
kernel, $\Vf = \ker(\IIH) = \{ v \in V \,:\, \IIH v = 0 \}$.
\begin{assumption}[Quasi-interpolation operator]
  \label{ass:qio}
  We require the quasi-interpolation operator to satisfy:
\begin{enumerate}
\item $\IIH$ is a linear projection onto $\VH$,
\item there exists a constant $C_{\Cint}$ independent of $H$ such that for
  all $v \in V$ and all $T \in \triH$, it holds
    \begin{equation}
      H^{-1}\| v - \IIH v\|_{L^2(T)} + \| \nabla (v - \IIH v) \|_{L^2(T)} \le C_{\Cint} \| \nabla v \|_{U(T)}.
    \end{equation}
    Here we define
    \begin{equation*}
      U(T) = \bigcup \{T' \in \triH \,:\, \overline T \cap \overline {T'} \ne \emptyset \},
    \end{equation*}
    i.e., the union of all neighboring elements to $T$.
\end{enumerate}
\end{assumption}

There are many possible options for choosing $\IIH$ and consequently
defining the \fine{} space $\Vf$. We define the numerical method in
an abstract setting where $\IIH$ only satisfies
Assumption~\ref{ass:qio} and postpone to define a concrete instance of
$\IIH$ to Sections~\ref{sec:IH} and \ref{sec:num}.

\subsection{Localized orthogonal decomposition}
\label{sec:localized}
Given a \fine{} space we define the \emph{non-local corrector operator}
$Q : V \to \Vf$ as an $A$-weighted Ritz-projection onto the \fine{}
space, find $Qv \in \Vf$, such that
\begin{equation}
  \label{eq:idealcorrector}
  a(Qv,\vf) = a(v,\vf)
\end{equation}
for all $\vf \in \Vf$. We introduce the multiscale space
$\VmsH = (\id - Q)V_H = \{v - Qv\,:\,v\in V_H\}$ and note that
$\dim(\VmsH) = \dim(V_H)$. From \eqref{eq:idealcorrector} we observe
that $\VmsH$ and $\Vf$ are orthogonal in the $a$-scalar product. Using
the low-dimensional $\VmsH$ as test and trial space in a standard
Galerkin method yields the following \emph{non-local multiscale method}, find
$\umsH \in \VmsH$, such that
\begin{equation}
  \label{eq:idealmultiscale}
  a(\umsH,v) = (f, v)
\end{equation}
for all $v \in \VmsH$. Galerkin orthogonality yields
$a(u - \umsH, v) = 0$, i.e.\ the error $\ufH = u - \umsH$ is in the
\fine{} space $\Vf$. Using that $\IIH v = 0$ for $v \in \Vf$, we get
an error bound (in energy norm, $\III{\cdot}^2 = a(\cdot, \cdot)$),
for the non-local multiscale method,
\begin{equation}
  \begin{aligned}
    \label{eq:idealaproiri}
  \III{\ufH}^2 & = a(\ufH, \ufH) = (f, \ufH - \IIH \ufH) \\
   & \le \|f\|_{L^2(\Omega)} \|\ufH - \IIH \ufH\|_{L^2(\Omega)} \le C_{\rho, \Cint} \alpha^{-1/2} H \|f\|_{L^2(\Omega)} \III{\ufH}.
  \end{aligned}
\end{equation}
(Here, and in the remainder of the paper, $C$ denotes a function that
depends on the variables or quantities listed in its subscript index
list.) We note that this a priori error bound is independent of the
solution regularity. In a practical implementation of this non-local
multiscale method, the multiscale space $\VmsH$ is spanned by a basis
$\{\phi_i - Q\phi_i\}_i$ where $\{\phi_i\}_i$ is the nodal basis of
$V_H$. This requires the non-local corrector problem
\eqref{eq:idealcorrector} to be solved for each basis function
$\phi_i$ in the \fine{} space, which is roughly as costly as solving
the original problem for each basis function. However, in
\cite{MaPe14} it was proven that $Q\phi_i$ exhibits exponential decay
from the support of $\phi_i$ and the corrector problems allow for
localization with a small sacrifice in accuracy.

For the localization, we define patches
$U_k(\omega) \subset \Omega$, where $0 \le k \in \mathbb{N}$ and
$\omega \subset \Omega$. With trivial case $U_0(\omega) = \omega$, $U_k(\omega)$ is defined by
the recursive relation
\begin{equation*}
  U_{k+1}(\omega) = \bigcup \{T' \in \triH : \overline{U_{k}(\omega)} \cap \overline {T'} \ne \emptyset \}.
\end{equation*}
We generalize the notation presented in Assumption~\ref{ass:qio} and
let $U(\omega) = U_1(\omega)$. In particular, if
$\omega = T \in \triH$, then $U_k(T)$ is a $k$-layer element patch
around $T$. If $\omega = \{z\}$ and $z \in \nodesH$, then (abusing
notation) $U_k(z)$ is a $k$-layer node patch around $z$. See
Figure~\ref{fig:patch} for an illustration of element patches.
\begin{figure}[]
  \centering
  \begin{subfigure}{.4\textwidth}
    \centering
    \includegraphics[width=3cm, clip=true, trim=3cm 1.2cm 3cm 1.2cm]{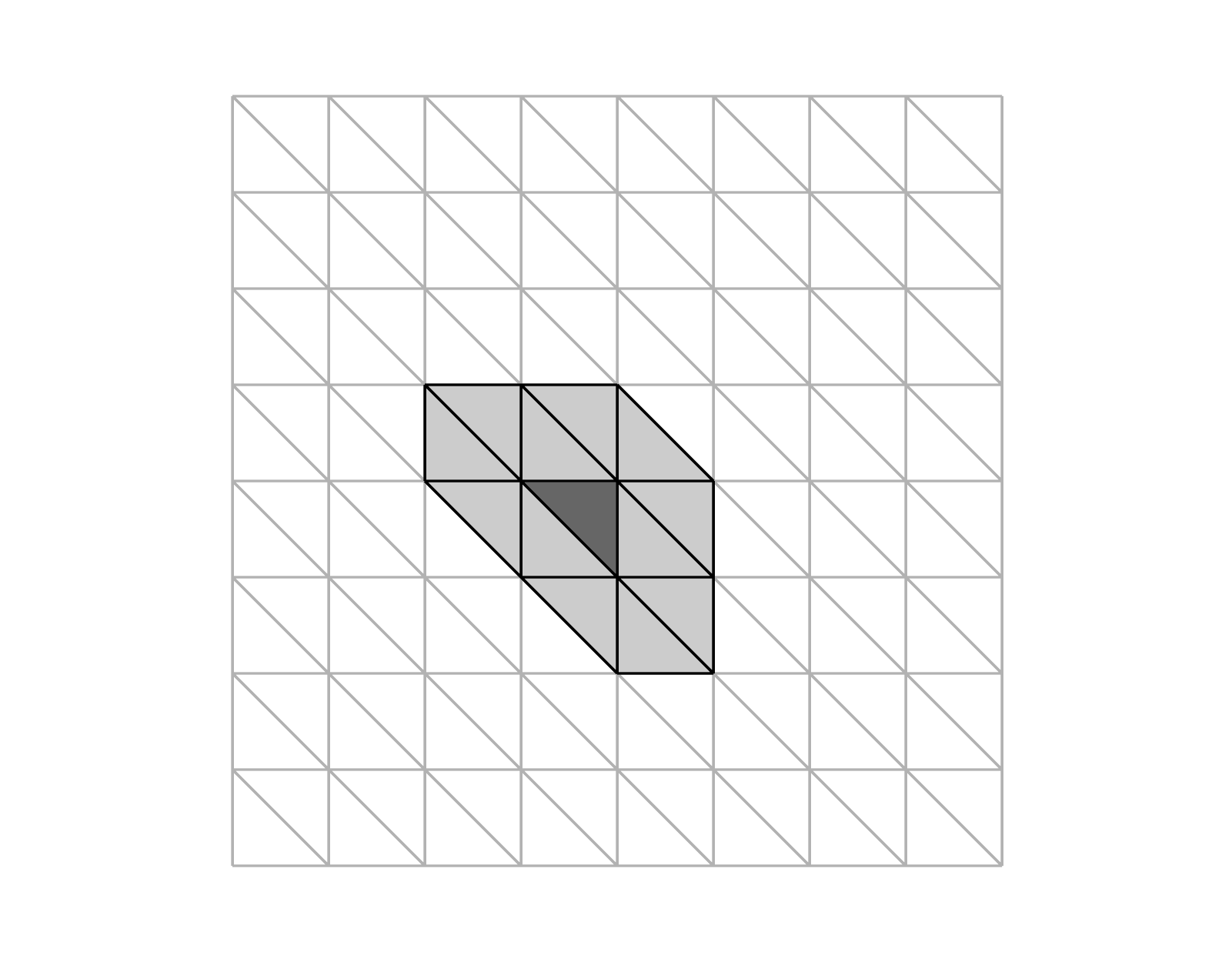}
    \caption{One-layer element patch, $k = 1$.}
  \end{subfigure}
  \hspace{1em}
  \begin{subfigure}{.4\textwidth}
    \centering
    \includegraphics[width=3cm, clip=true, trim=3cm 1.2cm 3cm 1.2cm]{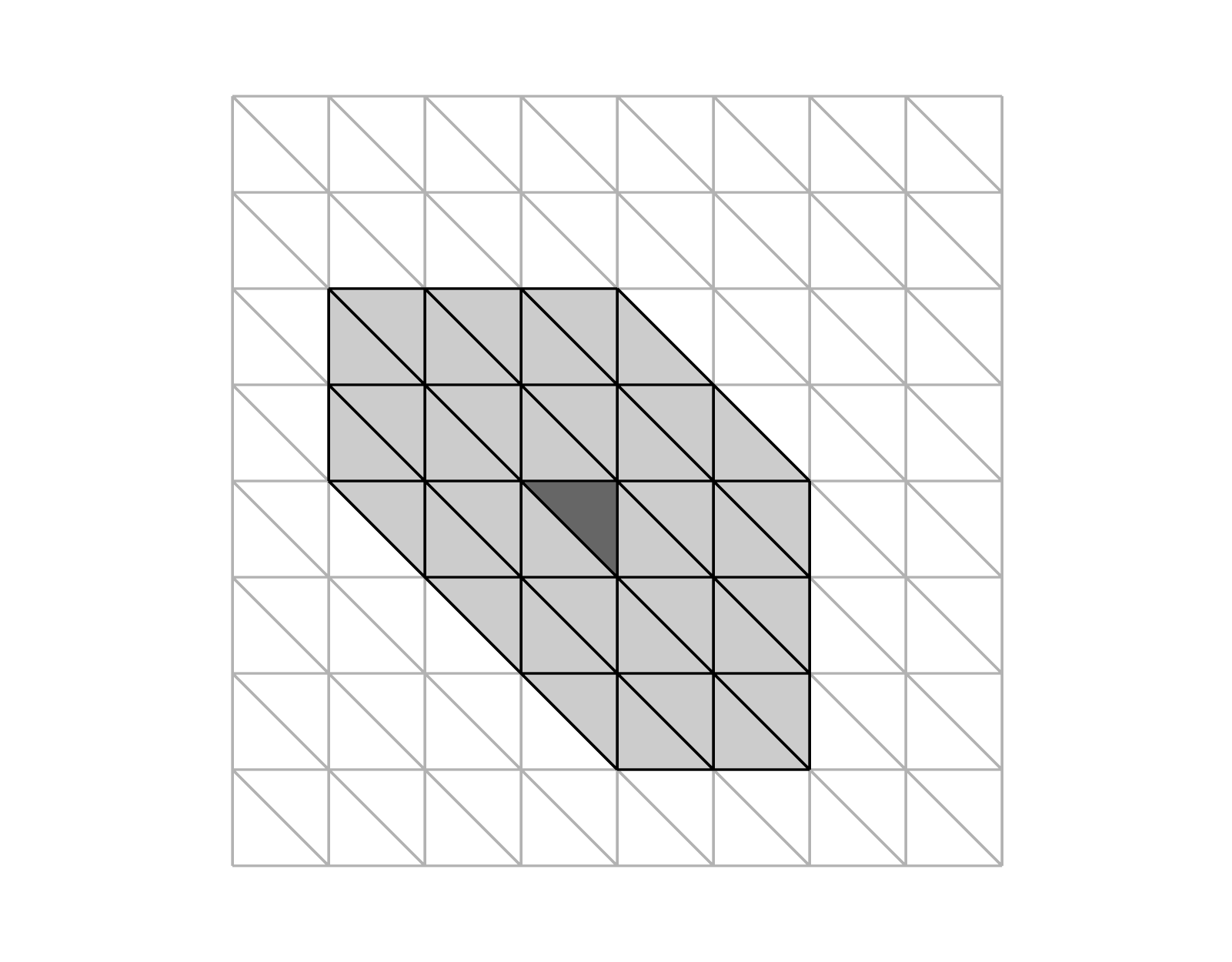}
    \caption{Two-layer element patch, $k = 2$.}
  \end{subfigure}
  \caption{Illustration of $k$-layer element patches. Dark gray is
    $T$.  Light gray is $U_k(T)$.}
  \label{fig:patch}
\end{figure}

Define localized \fine{} spaces
\begin{equation*}
  \Vf(U_k(T)) = \{ v \in \Vf \,:\, v|_{\Omega \setminus U_k(T)} = 0\},
\end{equation*}
consisting of \fine{} functions which are zero outside patches. The
localized corrector operator
$Q_{k}v = \sum_{T \in \triH} Q_{k,T}v$ is a sum of localized patch
corrector operators $Q_{k,T} : V \to \Vf(U_k(T))$, defined by, find
$Q_{k,T} \in \Vf(U_k(T))$, such that
\begin{equation}
  \label{eq:patchcorrector}
  a(Q_{k,T}v,\vf) = \int_T A \nabla v \nabla \vf,
\end{equation}
for all $\vf \in \Vf(U_k(T))$. Note that the problem
\eqref{eq:patchcorrector} is posed on the subdomain $U_k(T)$
only. This localization is what is behind the name localized
orthogonal decomposition. In a practical setting, $Q_{k,T}\phi_i$ is
solved for each coarse basis function $\phi_i$ and triangle $T$. Since
$\phi_i$ has local support, the full localized corrector $Q_k \phi_i$
is a sum of only a few $Q_{k,T}\phi_i$. Now, a localized multiscale
space $\VmsHk = (\id - Q_k)V_H$ is defined, spanned by a basis
$\{\phi_i - Q_k\phi_i\}_i$. The corresponding \emph{localized
  multiscale method} is, find $\umsHk \in \VmsHk$, such that
\begin{equation}
  \label{eq:localizedmultiscale}
  a(\umsHk,v) = (f, v)
\end{equation}
for all $v \in \VmsHk$.

The localization introduces an error that depends on the patch size
$k$. We can decompose the non-local multiscale solution
$\umsH = (\id - Q)u_H$ for a unique $u_H \in V_H$.  By Galerkin
orthogonality using \eqref{eq:problem} and
\eqref{eq:localizedmultiscale}, for any $v \in \VmsHk$, in particular
$v = (\id - Q_k)u_H$, we have:
\begin{equation}
  \label{eq:localizationerror}
  \begin{aligned}
    \III{u - \umsHk} \le \III{u - v} = \III{\uf} + \III{(Q - Q_k)u_H} \le \III{\uf} + C_\rho\alpha^{-1}k^{d/2}\theta(\alpha)^{k}\|f\|_{L^2(\Omega)},\\
  \end{aligned}
\end{equation}
where $0 < \theta(\alpha) < 1$ depends on the contrast $\alpha$ (see
\cite[Theorem 4.6]{MaPe14} or \cite[Theorem 3.7]{HeMa14} for the last
step.) The first term is the error from the non-local multiscale method
\eqref{eq:idealaproiri}, and the second term is the error from
localization, which decays exponentially with patch size $k$. We note
that for a given contrast, we can choose $k \approx \log H^{-1}$ to
make the term $\III{(Q-Q_k)u_H}$ of the same order as $\III{\uf}$,
i.e.\ the error can be kept small by choosing appropriate patch
sizes. However, we also note that the decay rate with respect to $k$
varies with $\alpha$. This is the contrast problem that we address in
this paper. The effect of the contrast problem can be observed in the
section of numerical experiments, Section~\ref{sec:num}, for
quasi-interpolation operator $\IHsz$.

\subsection{Right hand side correction}
\label{sec:rhscorrection}
Although the non-local multiscale method error $\III{\uf}$ converges
with $H$ independent of the regularity of the solution $u$, we are
interested in discarding this error to clearly distinguish the
contrast dependency due to localization, i.e.\ how strongly $\alpha$
influences $\theta(\alpha)$ and the total error. This section
describes how to compute a localized estimate $\ufk$ of $\uf$ to the
same accuracy as the localization.

Introduce the right hand side correction operator $R : V \to \Vf$,
\begin{equation*}
  a(Rv, \vf) = (v, \vf),
\end{equation*}
for all $\vf \in \Vf$. Then the error $\uf = Rf$. This problem can be
split into a number of localized problems, one for each coarse
triangle of which $f$ has support. We construct right hand side
correction operators $R_{k,T} : V \to \Vf(U_k(T))$, find
$R_{k,T}f \in \Vf(U_k(T))$ such that for all $\vf \in \Vf(U_k(T))$,
\begin{equation}
  \label{eq:localizedrhs}
  a(R_{k,T} f, \vf) = \int_T f \vf,
\end{equation}
in analogy with the localized corrector problems.  Note that if $f$
has support only in a few triangles, only a few problems of the kind
above needs to be solved.  The full localized right hand side
correction is $\ufk = R_kf = \sum_{T \in \triH} R_{k,T}f$. It is again
possible to show exponential decay of the localization error in terms
of $k$,
\begin{equation*}
  \III{\uf - \ufk} \le C\alpha^{-1}k^{d/2}\theta(\alpha)^{k}\III{\uf}.
\end{equation*}
Now, we have the following method (localized multiscale method with
right hand side correction)
\begin{enumerate}
\item Compute $R_{k,T} f$ in
  \eqref{eq:localizedrhs} (localized patch problems) for all $T$ for
  which $f|_T \ne 0$ and let $\ufk = \sum_{T \in \triH} R_{k,T}f$.
\item Compute all basis correctors $Q_{k,T} \phi_i$ by
  \eqref{eq:patchcorrector}
  and let
  $\{ \phi_i - Q_{k,T} \phi_i \}_i$ span $\VmsHk$.
\item Find $\umsLHk \in \VmsHk$,
  such that
  \begin{equation}
    \label{eq:localizedcorrectedmultiscale}
    a(\umsLHk,v) = (f, v) - a(\ufk, v)
  \end{equation}
  for all $v \in \VmsHk$.
\item The solution is $u_{H,k} = \umsLHk + \ufk$.
\end{enumerate}
By Galerkin orthogonality using \eqref{eq:problem} and
\eqref{eq:localizedcorrectedmultiscale}, for any $v \in \VmsHk$, in
particular $v = (\id - Q_k)u_H$, we have:
\begin{equation*}
  \label{eq:rhscorrectederror}
  \begin{aligned}
    \III{u - u_{H,k}} & \le \III{u - v - \ufk} \le \III{\uf - \ufk} + \III{(Q - Q_k)u_H} \\
    & \le C\alpha^{-1}k^{d/2}\theta(\alpha)^{k}(1 + \alpha^{-1/2}H)\|f\|_{L^2},
  \end{aligned}
\end{equation*}
i.e.\ the factor $\theta(\alpha)$ influences all error terms, and the
contrast dependency can be easily studied.

\newcommand{\gridcolor}{darkgray}

\newcommand{\gridone} {%
  \foreach \x in {0} {
    \foreach \y in {0} {
      \begin{scope}[shift={(\x,\y)}]
        \draw[\gridcolor] (0,0) grid (1,1);
        \draw[\gridcolor] (0,1) -- (1,0);
      \end{scope}
    }
  }
}

\newcommand{\gridtwo} {%
  \foreach \x in {0,1} {
    \foreach \y in {0,1} {
      \begin{scope}[shift={(\x,\y)}]
        \draw[\gridcolor] (0,0) grid (1,1);
        \draw[\gridcolor] (0,1) -- (1,0);
      \end{scope}
    }
  }
}
\section{Geometry induced quasi-interpolation operator $\IH$}
\label{sec:IH}
This section defines a novel interpolation operator $\IH$. Several
other operators are studied in the numerical experiments in
Section~\ref{sec:num} for evaluation on a series of high contrast
coefficients. However, the theoretical results presented in this paper
all concerns the operator $\IH$ defined in this section.

The main purpose of $\IH$ is to give rise to a space $\Vf$ in which
correctors $Q v$ and right hand side corrections $R f$ decay fast with
the distance from the support of $v$ and $f$, in order to give a small
localization error. Numerical experiments show that the decay rate for
classical Cl\'ement-type interpolation operators can be very low for
high contrast coefficients. In the following, we give a heuristic
argument to why the decay rate is low at high contrast. This argument
also motivates the way $\IH$ is defined. We stress that the purpose of
this argument is to illustrate the idea behind the construction of
$\IH$, and it is not necessary for the results presented in this paper
to hold.

Consider the corrector $Qv_T$ from \eqref{eq:idealcorrector} of a
function $v_T$ that has support only in triangle $T$. For $d=1$, it is
well-known (see e.g.\ \cite{HuSa07}) that total element localization
(i.e.\ $Qv_T|_{\partial T} = 0$) is possible by choosing the nodal
interpolation operator to define $\Vf$. However, total element
localization does not seem to be possible in higher dimensions. Thus,
for the sake of this argument, we neglect the influence of $\IIH$ on
$Qv_T|_{\partial T}$ and set $Qv_T|_{\partial T} = g$ to study the
decay of $Qv_T$ outside $T$.  Under this assumption, $Qv_T$ is the
minimizer
\begin{equation*}
  \|A^{1/2} \nabla Qv_T\|_{L^2(\Omega \setminus T)} \le \|A^{1/2} \nabla \vf \|_{L^2(\Omega \setminus T)} \qquad\text{for all } \vf \in \Vf,
\end{equation*}
with boundary conditions on $\partial \Omega$ and
$Qv_T|_{\partial T} = g$.  We see from this minimization problem that
derivatives $\nabla Qv_T|_{\Omega^1}$ in $\Omega^1$ are more heavily
penalized than derivatives $\nabla Qv_T|_{\Omega^\alpha}$ in
$\Omega^\alpha$, if the contrast is high. Thus, classical Cl\'ement-type
interpolation operators $\IIH$ that use (possibly weighted) averages
or local projections over full node patches as node variables, allows
for the possibility to satisfy the requirement $\IIH Qv_T = 0$ (i.e.\
$Qv_T \in \Vf$) by large variations of $Qv_T|_{\Omega^\alpha}$,
keeping $Qv_T|_{\Omega^1}$ fairly constant (or slowly decaying to
satisfy Dirichlet boundary conditions). This effect causes slow decay
through channels where $A$ is large.

The idea behind the geometry induced operator $\IH$ is to dedicate a
set of node variables (called class I nodes below) that do the
averaging or local projections on subdomains of $\Omega^1$ only. Those
subdomains are always connected subsets of $\Omega^1$ in order to
admit a local Poincar\'e-type inequality. Class I nodes forces decay
within $\Omega^1$ since large variations in $\Omega^\alpha$ cannot
help to satisfy $\IH Qv_T = 0$ if some node variables are defined in
terms of values from $\Omega^1$ only. To construct $\IH$, there are
certain conditions on the placement of nodes in relation to the
coefficient. Although some of these conditions can be relaxed, the
basic requirement (which is made precise below) is that all inclusions
of $\Omega^1$ should contain nodes frequently enough. With enough
nodes to dedicate at least one node to all connected inclusions, it is
possible to construct $\IH$.

A similar idea is that of the $A$-weighted projective
quasi-interpolation used in \cite{PeSc16}, in the sense that the
projection integral of the latter operator (which is unconditionally
taken over the full node patch) is weighted by $A$, so that variations
within $\Omega^\alpha$ becomes less significant in the
interpolation. However, instead of defining the operator in terms of
the geometry of $\Omega^1$ within the node patch, there is a
quasi-monotonicity assumption of the coefficient distribution for
contrast independent localization error.

The idea of selecting integration domain for the node variables is
conveyed in this paper by considering a particular operator $\IH$
based on the Scott--Zhang-type node variable. This is, however, only a
choice made in order to perform a concrete error analysis. The main
idea can be applied also for other node variables. As an example of
this, we will (without theoretical analysis) apply the idea also to
the $A$-weighted projective quasi-interpolation operator in the
numerical experiments.

\subsection{Scott--Zhang type node variables}
\label{sec:scottzhang}
We briefly review Scott--Zhang type node variables \cite{ScZh90}. We
use $i$ and $j$ to index basis functions and denote by $\phi_i$ the
nodal basis function of $S_H$ associated with node $z_i$. We
introduce $N_i$ as a Scott--Zhang-type node variable corresponding to
node $z_i$. To each node variable, we associate a domain
$\sigma_i \subset \Omega$ that includes $z_i$. Based on this
domain, we define an $L^2(\sigma_i)$-dual basis $\psi_i$,
satisfying, for all basis functions $\phi_j$ of $S_H$,
\begin{equation}
  \label{eq:psidef}
  \int_{\sigma_i} \psi_i \phi_j = \delta_{ij} = 
  \begin{cases}
    1 & \quad i=j, \\ 
    0 & \quad i\ne j. \\ 
  \end{cases}
\end{equation}
This dual basis is used to define the node variable,
\begin{equation}
  \label{eq:Ndef}  
  N_i(v) = \int_{\sigma_i} \psi_i v.
\end{equation}

\subsection{Geometry and mesh assumptions}
In order to define $\IH$, we impose the following assumptions on the
geometry of $\Omega^1$ and the mesh. This is the formal condition for
frequent enough node placement within $\Omega^1$.
\begin{assumption}[Existence of a covering set of subdomains of
  $\Omega^1$ with diameter and Poincar\'e constants uniformly bounded by
  $H$]
  \label{ass:localpoincare}
  For a fixed $H$, there is a set
  $\{\omega^1_1, \ldots, \omega^1_{\MW}\}$ of $\MW$ open and possibly
  overlapping subdomains of $\Omega^1$, that satisfy the following
  conditions:
  \begin{enumerate}
  \item They cover exactly $\Omega^1$, i.e.\
    $\overline{\bigcup_{i=1}^{\MW} \omega^1_{i}} = \overline{\Omega^1}$.
  \item For each $\omega^1_i$, there is a free node $z_i \in \nodesfreeH$ such that $z_i \in \omega^1_i$.
  \item There is a constant integer $\Cdiam$ independent of $H$, such that  $\omega^1_{i} \subset U_{\Cdiam}(z_i)$.
  \item There is a constant $\Cpoinc$ independent of $H$, such that the following Poincar\'e inequalities hold,
    \begin{equation}
      \label{eq:poincare}
      \inf_{q\in \R} \|v - q\|_{L^2(\omega^1_{i})} \le \Cpoinc H \|\nabla v\|_{L^2(\omega^1_{i})},\qquad i =1,\ldots,\MW.
    \end{equation}
  \end{enumerate}
\end{assumption}
The constant $\Cpoinc$ carries information about the geometry of the
subdomains. Non-chunky subdomains give rise to larger $\Cpoinc$. The
constant $\Cdiam$ relates the diameter of the subdomains to the mesh
size (since we assume quasi-uniform meshes). Next, we illustrate by
two examples the relation between geometry of $\Omega^1$ and the
constants $\Cpoinc$, $\Cdiam$ and the mesh.

\newcommand{\shapebar}{%
  \fill[black] (-0.5,-0.05) rectangle (0.5,0.05);
}

\definecolor{iceberg}{rgb}{0.44, 0.65, 0.82}
\def \redcolor {iceberg}
\newcommand{\shapebarred}{%
  \fill[\redcolor] (-0.25,-0.05) rectangle (0.25,0.05);
}

\newcommand{\shapevbar}{%
  \fill[black] (-0.05,-0.5) rectangle (0.05,0.5);
}

\newcommand{\shapevbarred}{%
  \fill[\redcolor] (-0.05,-0.25) rectangle (0.05,0.25);
}

\newcommand{\shapecircle}{%
  \fill[black] (0,0) circle [radius=0.25];
}

\newcommand{\shapecirclered}{%
  \fill[\redcolor] (0,0) circle [radius=0.25];
}

\newcommand{\shapecross}{%
  \fill[black] (-0.5,-0.05) rectangle (0.5,0.05);
  \fill[black] (-0.05,-0.5) rectangle (0.05,0.5);
}

\newcommand{\shapecrosscircle}{%
  \fill[black] (-0.5,-0.05) rectangle (0.5,0.05);
  \fill[black] (-0.05,-0.5) rectangle (0.05,0.5);
  \fill[black] (0,0) circle [radius=0.25];
}

\newcommand{\shapecrosscirclered}{%
  \fill[\redcolor] (-0.5,-0.05) rectangle (0.5,0.05);
  \fill[\redcolor] (-0.05,-0.5) rectangle (0.05,0.5);
  \fill[\redcolor] (0,0) circle [radius=0.25];
}

\newcommand{\shapeblob}{%
  \fill[domain=0:2*pi,samples=500] plot ({deg(\x)}:{0.1+0.5*(abs(0.4*cos(\x r) + 0.2*sin(8*\x r) + 0.2*sin(3*\x r) + 0.2*cos(5*\x r)))});
}

\begin{example}[The constants $\Cpoinc$ and $\Cdiam$]
  In some applications, the subdomain $\Omega^1$ can be described as a
  union of similar shapes. Figure~\ref{fig:shapes} gives five examples
  of possible shapes scaled to unit diameter and presents bounds on
  the corresponding Poincar\'e constants. The Poincar\'e constants have
  been generously estimated using the results for convex and
  star-shaped domains in \cite[eq. (1.1), (1.2)]{Ve99}. For the
  examples shown, we note that the Poincar\'e constants for the shapes
  (a), (b) and (d) are bounded independently of the fine scale
  parameter $\epsilon$. The bound for shape (c) increases very slowly
  with $\epsilon$ and is reasonably bounded for most practical
  implementations.

  These shapes can be translated and rotated without affecting their
  Poincar\'e constants. Scaling them by a factor $\Cdiam H$ scales their
  Poincar\'e constants similarly. Thus, a union of these shapes
  (translated, rotated and scaled) can be used to construct subdomains
  $\Omega^1$ depicted in Figure~\ref{fig:omega1}.

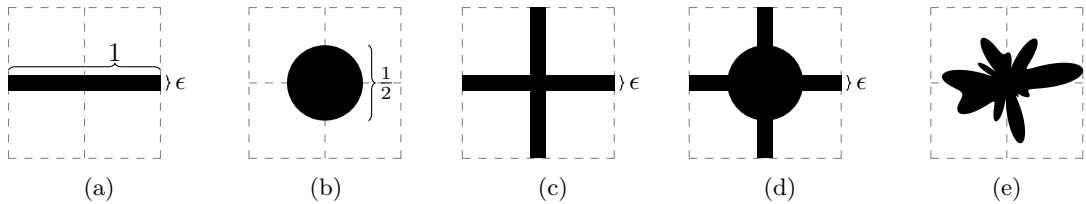
\begin{figure}[]
  \centering
  \begin{subfigure}{.15\textwidth}
    \centering
    \begin{tikzpicture}[scale=2]
      \draw[step=0.5, gray, dashed] (-0.5,-0.5) grid (0.5,0.5);
      \shapebar
      
      \draw[decoration={brace,raise=2pt, amplitude=+1.5pt},decorate]
      (0.5,0.05) -- node[right=2pt] {$\epsilon$} (0.5,-0.05);
      
      \draw[decoration={brace, raise=2pt, aspect=0.7},decorate]
      (-0.5,0.05) -- node[above=2pt,pos=0.7] {$1$} (0.5,0.05);
    \end{tikzpicture}%
    \caption{}
    \label{fig:shape_bar}
  \end{subfigure}
  \hspace{1em} 
  \begin{subfigure}{.15\textwidth}
    \centering
    \begin{tikzpicture}[scale=2]
      \draw[step=0.5, gray, dashed] (-0.5,-0.5) grid (0.5,0.5);
      \shapecircle
      
      \draw[decoration={brace,raise=2pt},decorate]
      (0.25,0.25) -- node[right=2pt] {$\frac{1}{2}$} (0.25,-0.25);
    \end{tikzpicture}%
    \caption{}
    \label{fig:shape_circle}
  \end{subfigure}
  \hspace{1em} 
  \begin{subfigure}{.15\textwidth}
    \centering
    \begin{tikzpicture}[scale=2]
      \draw[step=0.5, gray, dashed] (-0.5,-0.5) grid (0.5,0.5);
      \shapecross
      
      \draw[decoration={brace,raise=2pt, amplitude=+1.5pt},decorate]
      (0.5,0.05) -- node[right=2pt] {$\epsilon$} (0.5,-0.05);
      
    \end{tikzpicture}%
    \caption{}
  \end{subfigure}
  \hspace{1em} 
  \begin{subfigure}{.15\textwidth}
    \centering
    \begin{tikzpicture}[scale=2]
      \draw[step=0.5, gray, dashed] (-0.5,-0.5) grid (0.5,0.5);
      \shapecrosscircle
      
      \draw[decoration={brace,raise=2pt, amplitude=+1.5pt},decorate]
      (0.5,0.05) -- node[right=2pt] {$\epsilon$} (0.5,-0.05);
    \end{tikzpicture}%
    \caption{}
    \label{fig:shape_crosscircle}
  \end{subfigure}
  \hspace{1em} 
  \begin{subfigure}{.15\textwidth}
    \centering
    \begin{tikzpicture}[scale=2]
      \draw[step=0.5, gray, dashed] (-0.5,-0.5) grid (0.5,0.5);

      \shapeblob;

    \end{tikzpicture}%
    \caption{}
  \end{subfigure}

  \caption{Five examples of shapes: (a) convex domain $p \le \pi^{-1}$, (b) convex domain $p \le (2\pi)^{-1}$, (c) star-shaped $p \le \max(5\pi^{-1}, \log(\epsilon^{-1})^{1/2})$, (d) chunky star-shaped $p \le 5\pi^{-1}$, (e) polar function star-shaped, point $(r,\varphi)$ in shape if $r \le f(\varphi)$ with $\epsilon \le f(\varphi) \le 1$, then $p \le \max(5\pi^{-1}, \log(\epsilon^{-1})^{1/2}).$ }
  \label{fig:shapes}

\end{figure}
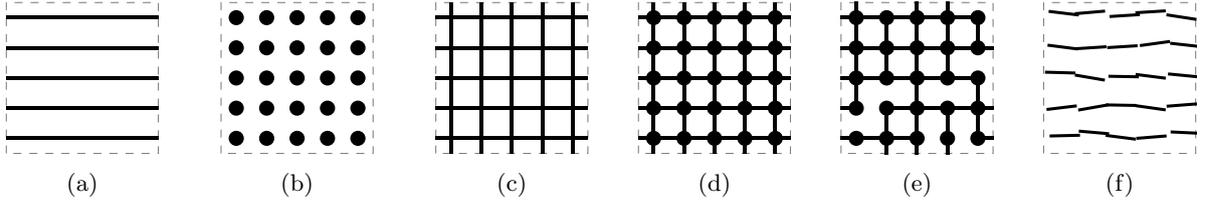
\begin{figure}[]
  \centering
  \begin{subfigure}{.13\textwidth}
    \centering
    \begin{tikzpicture}[scale=0.4]
      \clip (0.48,0.48) rectangle (5.52,5.52);
      \draw[gray, dashed] (0.5,0.5) rectangle (5.5,5.5);
      \foreach \x in {1,2,...,5}
      \foreach \y in {1,2,...,5}
      \begin{scope}[shift={(\x,\y)}]
        \shapebar
      \end{scope};;
    \end{tikzpicture}%
    \caption{}
  \end{subfigure}
  \hspace{1em} 
  \begin{subfigure}{.15\textwidth}
    \centering
    \begin{tikzpicture}[scale=0.4]
    \clip (0.48,0.48) rectangle (5.52,5.52);
    \draw[gray, dashed] (0.5,0.5) rectangle (5.5,5.5);
    \foreach \x in {1,2,...,5}
    \foreach \y in {1,2,...,5}
    \begin{scope}[shift={(\x,\y)}]
      \shapecircle
    \end{scope};;
  \end{tikzpicture}%
    \caption{}
  \end{subfigure}
  \hspace{1em} 
  \begin{subfigure}{.13\textwidth}
    \centering
    \begin{tikzpicture}[scale=0.4]
      \clip (0.48,0.48) rectangle (5.52,5.52);
      \draw[gray, dashed] (0.5,0.5) rectangle (5.5,5.5);
      \foreach \x in {1,2,...,5}
      \foreach \y in {1,2,...,5}
      \begin{scope}[shift={(\x,\y)}]
        \shapecross
      \end{scope};;
    \end{tikzpicture}%
    \caption{}
  \end{subfigure}
  \hspace{1em} 
  \begin{subfigure}{.13\textwidth}
    \centering
    \begin{tikzpicture}[scale=0.4]
      \clip (0.48,0.48) rectangle (5.52,5.52);
      \draw[gray, dashed] (0.5,0.5) rectangle (5.5,5.5);
      \foreach \x in {1,2,...,5}
      \foreach \y in {1,2,...,5}
      \begin{scope}[shift={(\x,\y)}]
        \shapecrosscircle
      \end{scope};;
    \end{tikzpicture}%
    \caption{}
    \label{fig:omega1_crosscircle}
  \end{subfigure}
  \hspace{1em} 
  \begin{subfigure}{.13\textwidth}
    \centering
    \pgfmathsetseed{1}%
    \begin{tikzpicture}[scale=0.4]
      \clip (0.48,0.48) rectangle (5.52,5.52);
      \draw[gray, dashed] (0.5,0.5) rectangle (5.5,5.5);
      \foreach \x in {1,2,...,5} {
        \foreach \y in {1,2,...,5} {
          \begin{scope}[shift={(\x,\y)}]
            \shapecircle;
            \pgfmathparse{rnd}
          \end{scope}
        }
      }
      \foreach \x in {0,1,2,...,5} {
        \foreach \y in {0,1,2,...,5} {
          \begin{scope}[shift={(\x+0.5,\y)}]
            \pgfmathparse{rnd}
            \ifdim \pgfmathresult pt < 0.7pt \shapebar \else {} \fi;
          \end{scope}
        }
      }
      \foreach \x in {0,1,2,...,5} {
        \foreach \y in {0,1,2,...,5} {
          \begin{scope}[shift={(\x,\y+0.5)}]
            \pgfmathparse{rnd}
            \ifdim \pgfmathresult pt < 0.7pt \shapevbar \else {} \fi;
          \end{scope}
        }
      }
    \end{tikzpicture}%
    \caption{}
  \end{subfigure}
  \hspace{1em} 
  \begin{subfigure}{.13\textwidth}
    \centering

    \begin{tikzpicture}[scale=0.4]
      \clip (0.48,0.48) rectangle (5.52,5.52);
      \draw[gray, dashed] (0.5,0.5) rectangle (5.5,5.5);
      \foreach \x in {1,2,...,5}
      \foreach \y in {1,2,...,5}
      \begin{scope}[shift={(\x,\y)}]
        \pgfmathparse{rnd}
        \begin{scope}[shift={(0.2*\pgfmathresult,0)}]
          \pgfmathparse{rnd}
          \begin{scope}[shift={(0,0.2*\pgfmathresult)}]
            \pgfmathparse{rnd}
            \begin{scope}[rotate=20*\pgfmathresult-10]
              \shapebar
            \end{scope}
          \end{scope}
        \end{scope}
      \end{scope};;
    \end{tikzpicture}%
    \caption{}

  \end{subfigure}
  \caption{Six examples of coefficients. $\Omega^1$ is colored black
    and is a union of (possibly overlapping) subdomains like the ones
    presented in Figure~\ref{fig:shapes}. (a--d) Periodic using a
    single shape, (e) non-periodic using several shapes, (f)
    non-periodic with randomly rotated and translated shapes.}
  \label{fig:omega1}
\end{figure}
\end{example}

\begin{example}[Mesh and subdomain matching]
  Consider the coefficient in Figure~\ref{fig:omega1_crosscircle} and
  a mesh with mesh size $H$ with nodes in all circle centers, see
  Figure~\ref{fig:matching_2H}. We can decompose $\Omega^1$ into
  subdomains of the shape presented in
  Figure~\ref{fig:shape_crosscircle} (after translation and
  $H$-scaling). One such subdomain is colored blue in
  Figure~\ref{fig:matching_2H}.  This choice of subdivision and mesh
  fulfills Assumption~\ref{ass:localpoincare}.

  Next, consider the refined mesh in Figure~\ref{fig:matching_H}. The
  decomposition of $\Omega^1$ into subdomains can now be done in a way
  yielding better Poincar\'e constants than for the previous mesh, using
  the shapes in Figure~\ref{fig:shape_bar} and
  \ref{fig:shape_circle}. Three subdomains
  $\omega^1_i$, $\omega^1_j$, and $\omega^1_k$ are colored blue in the figure.

  As the mesh is refined further, it is easy to see that it is
  possible to define subdomains satisfying the two assumptions so that
  their Poincar\'e constants can be bounded by $\Cpoinc$ and their
  diameter by $\Cdiam H$ for some $\Cpoinc$ and $\Cdiam$ independent
  of $H$.
  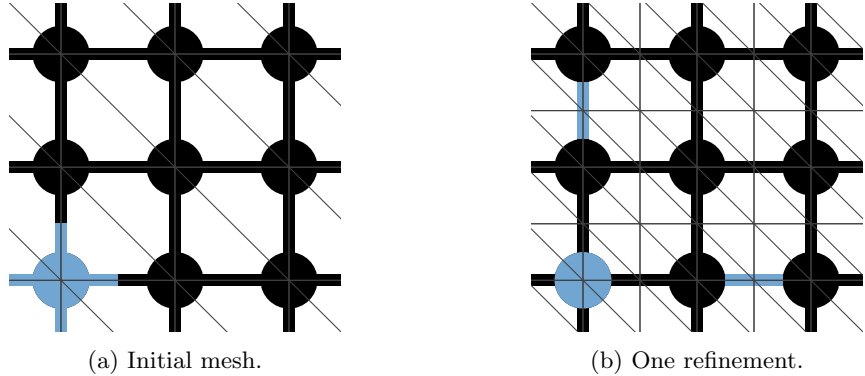
\begin{figure}
    \centering
    \begin{subfigure}{.4\textwidth}
      \centering
      \begin{tikzpicture}[scale=1.5]
        \clip (1.55,1.55) rectangle (4.45,4.45);
        \foreach \x in {2,3,4} {
          \foreach \y in {2,3,4} {
            \begin{scope}[shift={(\x,\y)}]
              \shapecrosscircle
            \end{scope}
          }
        };

        \begin{scope}[shift={(2,2)}]
          \shapecrosscirclered
        \end{scope}

        \foreach \x in {1,2,3,4} {
          \foreach \y in {1,2,3,4} {
            \begin{scope}[shift={(\x,\y)}]
              \draw[\gridcolor] (0,0) grid (1,1);
              \draw[\gridcolor] (0,1) -- (1,0);
            \end{scope}
          }
        }
      \end{tikzpicture}%
      \caption{Initial mesh.}
      \label{fig:matching_2H}
    \end{subfigure}%
    \hspace{1em}
    \begin{subfigure}{.4\textwidth}
      \centering
      \begin{tikzpicture}[scale=1.5]
        \clip (1.55,1.55) rectangle (4.45,4.45);
        \foreach \x in {2,3,4} {
          \foreach \y in {2,3,4} {
            \begin{scope}[shift={(\x,\y)}]
              \shapecrosscircle
            \end{scope}
          }
        };

        \begin{scope}[shift={(2,2)}]
          \shapecirclered
        \end{scope}

        \begin{scope}[shift={(3.5,2)}]
          \shapebarred
        \end{scope}

        \begin{scope}[shift={(2,3.5)}]
          \shapevbarred
        \end{scope}

        \foreach \x in {2,3,...,9} {
          \foreach \y in {2,3,...,9} {
            \begin{scope}[scale=0.5,shift={(\x,\y)}]
              \draw[\gridcolor] (0,0) grid (1,1);
              \draw[\gridcolor] (0,1) -- (1,0);
            \end{scope}
          }
        }
      \end{tikzpicture}%
      \caption{One refinement.}
      \label{fig:matching_H}
    \end{subfigure}
    \caption{Subdomain $\Omega^1$ is black and the triangular mesh is
      dark gray. In (a) one subdomain $\omega^1_i$ is blue and in (b)
      three subdomains $\omega^1_i$, $\omega^1_k$, and $\omega^1_k$
      are blue.}
  \end{figure}

\end{example}

\subsection{Definition of $\IH$}
\label{sec:definition}
The definition of $\IH$ is based on the Scott--Zhang node variables
reviewed above, but with particular choices of integration domains
$\sigma_i$. All free nodes are partitioned into two classes, and the
domain is determined based on the class of the corresponding
node. The class I nodes are dedicated to guarantee fast decay within
$\Omega^1$ and are thus associated with the subdomains $\omega^1_i$
defined above.
\begin{enumerate}[I.]
\item The first class $\nodesIH \subset \nodesfreeH$ is a selection of
  free nodes such that for each domain $\omega^1_i$ there is a
  $z_i \in \nodesIH$ satisfying $z_i \in \omega^1_i$. We number these
  nodes $z_i$ for $i=1,\ldots,M$, although they are not necessarily
  unique.
\item The second class $\nodesIIH = \nodesfreeH \setminus \nodesIH$
  consists of all remaining free nodes.
\end{enumerate}
Further, we define
$\Sigma_\delta(z) = \{\delta(x-z)+z \,:\, x\in U(z)\}$, which is a
node-centered $\delta$-scaling of the node patch of a node $z$, with
$0 < \delta \le 1$. Now,
\begin{equation*}
  \IHdelta v = \sum_{i \,:\, z_i \in \nodesfreeH } N_i(v)\phi_i
\end{equation*}
where $N_i$ is defined in \eqref{eq:Ndef}, and
\begin{equation*}
  \begin{aligned}
    \sigma_i &= U(z_i)\cap\omega^1_{i} && \text{ for } z_i \in \nodesIH, \\
    \sigma_i &= \Sigma_\delta(z_i) && \text{ for } z_i \in \nodesIIH. \\
  \end{aligned}
\end{equation*}
The case $\delta = 1/4$ defines $\IH = \IHdelta$.  See
Figure~\ref{fig:nodeclasses} for an illustration of the two classes
and their integration domain $\sigma_i$. Note that class I node
variables integrate only over $\Omega^1$ since
$\omega^1_i \subset \Omega^1$. The integration domain for class II
nodes are restricted to the $\delta$-scaling of the node patch for
technical reasons which are not fully understood, but appears
necessary for the proof of contrast independent localization error
(see Remark~\ref{rem:classIII} below). An interesting special case is
when $\delta = 1$, i.e.\ that the full node patch is used also for
class II nodes. Choosing $\delta = 1$ defines the operator
$\IHinf$. We study $\IH$ both analytically and numerically and include
$\IHinf$ only in the numerical experiments for comparison.

\begin{remark}
  \label{rem:classIII}
  The theoretical results below holds for a slightly more general
  operator (which will not be discussed outside this remark) than
  $\IH$. Introduce a third class for nodes $z_i$ for which there are
  no channels or inclusions in its node patch, i.e.\
  $U(z_i) \cap \Omega^1 = \emptyset$. For those nodes, it is possible
  to choose integration domain as full node patch,
  $\sigma_i = U(z_i)$, and still have the contrast independent
  localization error results hold. That is, except for class I nodes,
  only nodes close to inclusions or channels actually need the
  integration domain restriction to the $\delta$-scaled node patch. As
  we will see in Section~\ref{sec:num}, the numerical experiments for
  $\IHinf$ even suggest that the restriction to $\delta$-scaled node
  patches is not necessary for any node.
\end{remark}

\newcommand{\patch} {%
  \begin{scope}[shift={(1,2)}]
    \draw[\gridcolor] (0,0) grid (1,1);
  \end{scope}
  \begin{scope}[shift={(2,1)}]
    \draw[\gridcolor] (0,0) grid (1,1);
  \end{scope}
  \foreach \x in {1,2} {
    \foreach \y in {1,2} {
      \begin{scope}[shift={(\x,\y)}]
        \draw[\gridcolor] (0,1) -- (1,0);
      \end{scope}
    }
  }
}

\def \greencolor {DarkGoldenrod1}
\newcommand{\shapebargreen}{%
  \fill[\greencolor] (-0.5,-0.05) rectangle (0.5,0.05);
}

\newcommand{\omegaoneshape} {
  \begin{scope}[xscale=1.2]
    \shapebar;
  \end{scope}
  \begin{scope}[shift={(-0.1,0.1)},rotate=45]
    \shapebar;
  \end{scope}
  \begin{scope}[shift={(0.5,-0.5)}, xscale=1.2]
    \shapebar;
  \end{scope}
}

\newcommand{\omegaoneshapegreen} {
  \begin{scope}[shift={(0.5,-0.5)}, xscale=1.2]
    \shapebar;
  \end{scope}
  \begin{scope}[xscale=1.2]
    \shapebargreen;
  \end{scope}
  \begin{scope}[shift={(-0.1,0.1)},rotate=45]
    \shapebargreen;
  \end{scope}
}

\newcommand{\omegaoneshapetwo} {
  \begin{scope}[shift={(0.5,-0.5)}, xscale=1.2]
    \shapebar;
  \end{scope}
}

\begin{figure}[]
  \centering
  \begin{subfigure}{.3\textwidth}
    \centering
    \begin{tikzpicture}[scale=1.5]
      \clip (0.9, 0.9) rectangle (3.1, 3.1);
      \begin{scope}[scale=2,shift={(1,1)}]
        \omegaoneshape;
        \def \s {0.5};
        \clip (-\s,0) -- (-\s,\s) -- (0,\s) -- (\s,0) -- (\s, -\s) -- (0, -\s) -- (-\s, 0);
        \draw[densely dotted] (-\s,0) -- (-\s,\s) -- (0,\s) -- (\s,0) -- (\s, -\s) -- (0, -\s) -- (-\s, 0);
        \omegaoneshapegreen;
      \end{scope}
      \patch;
    \end{tikzpicture}%
    \caption{Class I ($\nodesIH$).}
    \label{fig:nodeclass_I}
  \end{subfigure}%
  \begin{subfigure}{.3\textwidth}
    \centering
    \begin{tikzpicture}[scale=1.5]
      \clip (0.9, 0.9) rectangle (3.1, 3.1);
      \begin{scope}[scale=2,shift={(1,1)}]
        \omegaoneshapetwo;
      \end{scope}
      \begin{scope}[shift={(2,2)}]
        \def \s {0.25};
        \clip (-\s,0) -- (-\s,\s) -- (0,\s) -- (\s,0) -- (\s, -\s) -- (0, -\s) -- (-\s, 0);
        \draw[densely dotted] (-\s,0) -- (-\s,\s) -- (0,\s) -- (\s,0) -- (\s, -\s) -- (0, -\s) -- (-\s, 0);
        \fill[\greencolor] (-1,-1) rectangle (1,1);
      \end{scope}
      \patch;
    \end{tikzpicture}%
    \caption{Class II ($\nodesIIH$).}
    \label{fig:nodeclasse_II}
  \end{subfigure}
  \caption{Integration domain for two nodes of class I and II,
    respectively. Black is $\Omega^1$. Yellow is $\sigma_i$.}
  \label{fig:nodeclasses}
\end{figure}
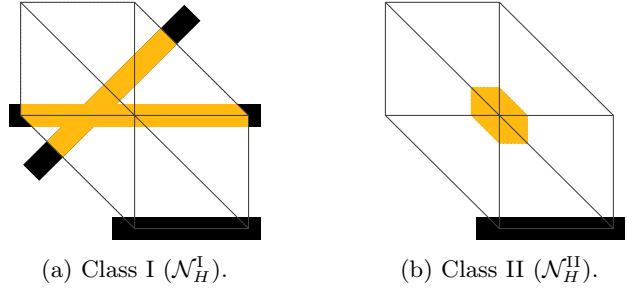

\subsection{Node variable stability}
For class I nodes, the integration domains $\sigma_i$ are restricted to
subsets of $\Omega^1$ and for class II nodes, they are restricted to
the $\delta$-scaling of the node patch. The $L^2$-stability of the
node variables $N_i$ is influenced by the corresponding
$\sigma_i$. Thus, the stability and approximation properties of the
full interpolation operator depends on the choice of $\delta$, the
geometry of $\sigma_i$ and hence also on $\Omega^1$. In this
subsection, we study the $L^2$-stability of $N_i(v)$ with respect to
the integration domains $\sigma_i$. The results in this section hold
also for general Scott--Zhang interpolation operators, i.e.\ for any
choice $\sigma_i \subset U(z_i)$.

For the remainder of this section, we consider a single node variable
$N_i$ and drop the index $i$ for all quantities associated with that
node. Let $n$ be the number of basis functions with support in
$\sigma$ and number those basis functions by $\phi_j$,
$j = 1,\ldots,n$.  Without loss of generality, we assume
$\phi_1(z) = 1$. From the definition of $N$, we note that
\begin{equation*}
|N(v)| \le \|\psi\|_{L^2(\sigma)} \|v\|_{L^2(\sigma)}
\end{equation*}
and focus on $\|\psi\|_{L^2(\sigma)}$.

The next lemma shows how the (normalized) geometry of $\sigma$ affects
the $L^2$-norm of the dual basis. An affine transform $F$ is used to
rescale the physical node patch to a patch of diameter independent of
$H$.
\begin{lemma}[Dual basis $L^2$-norm]
  \label{lem:psil2}
  Let $F(\hat x) = B\hat x + z$ for $B \in \R^{d \times d}$ with $\det(B) = H^d$, i.e.\
  any affine transformation with volume rescaling proportional to the
  volume of $U(z)$. Define
  ${\hat \phi}_j(\hat x) = \phi_j(F(\hat x))$ and
  ${\hat \sigma} = F^{-1}(\sigma)$.  Let
  ${\hat M}_{jk} = \int_{\hat \sigma} {\hat \phi}_k {\hat \phi}_j
  \dd \hat x$,
  for $j,k=1,\ldots,n$ and ${\hat M}^{11}$ be the square submatrix of
  ${\hat M}$ where first row and column have been removed. Then, for the dual
  basis function $\psi$, we have
  \begin{equation}
    \label{eq:kappa}
    \|\psi\|_{L^{2}(\sigma)} = H^{-d/2} \kappa
  \end{equation}
  with $\kappa^2 = \frac{\det({\hat M}^{11})}{\det({\hat M})}$.
\end{lemma}
\begin{proof}
  With the assumed node numbering, the definition of $\psi$ is
  \begin{equation}
    \label{eq:psideflocal}
    \int_{\sigma} \psi \phi_j = \delta_{1j}. 
  \end{equation}
  We can express $\psi = \sum_{k=1}^{n} \xi_k \phi_k$. Using the
  definition of $\psi$, we observe that
    $\|\psi\|_{L^2(\sigma)}^2 = \int_{\sigma} \psi \sum_{j=1}^{n}\xi_j \phi_j = \xi_1.$
  With $M_{jk} = \int_{\sigma} \phi_k \phi_j \dd x$ (and
  corresponding definition of $M^{11}$), the linear system of
  equations \eqref{eq:psideflocal} can be expressed as
  $M\cdot\xi = (1, 0, \ldots, 0)$, where
  $\xi = (\xi_1, \ldots, \xi_{d+1})$. By Cramer's rule we get
  \begin{equation*}
    \label{eq:xiest}
    \xi_1 = \frac{\det(M^{11})}{\det(M)} = \det(B)^{-1}\frac{\det({\hat M}^{11})}{\det({\hat M})} .
  \end{equation*}
\end{proof}
The next lemma shows that the stability constant never increases by
extending $\sigma$ within the node patch.
\begin{lemma}[Extending integration domain never increases dual basis $L^2$-norm]
  \label{lem:extend}
  If $\tilde \sigma \supset \sigma$, and $\tilde \psi$ is defined
  analogously to $\psi$, but with $\sigma$ replaced by
  $\tilde \sigma$ (and $n$ by $\tilde n$) then
  \begin{equation*}
    \|\tilde \psi\|_{L^{2}(\tilde \sigma)} \le \|\psi\|_{L^{2}(\sigma)}.
  \end{equation*}
\end{lemma}
\begin{proof}
  Let $\psi = \sum_{k=1}^{n} \xi_k \phi_k$ and
  $\tilde \psi = \sum_{k=1}^{\tilde n} \tilde \xi_k \phi_k$ (note that
  $n \le \tilde n$). Then we have that
  $\|\tilde \psi\|_{L^2(\tilde \sigma)}^2 = \tilde \xi_1$, and
  $\int_\sigma \tilde \psi \psi = \tilde \xi_1$. Using this, we get
  \begin{equation*}
    \|\tilde \psi\|_{L^{2}(\tilde \sigma)}^2 = \tilde \xi_1 = \int_\sigma \tilde \psi \psi \le \|\tilde \psi\|_{L^{2}(\sigma)} \|\psi\|_{L^{2}(\sigma)} \le \|\tilde \psi\|_{L^{2}(\tilde \sigma)} \|\psi\|_{L^{2}(\sigma)}.
  \end{equation*}
\end{proof}
The stability constant $\kappa$ hence depends on the shape of
$\hat \sigma$ in a node patch normalized coordinate system. We also
note that $\kappa$ is a computable quantity. Next, we illustrate the
relation between the shape of $\hat \sigma$ and $\kappa$ by studying a
few examples in 2D.

\begin{example}[Constant $\kappa$ for four geometries of $\sigma$ when $d=2$]
  \label{ex:kappa}
  Consider the four geometries of $\sigma$ (restricted to a single
  element and normalized to unit element size) parametrized by
  $\epsilon$ depicted in Figure~\ref{fig:sigmageom}. We are interested
  in the behavior of the stability constant $\kappa$ for small
  $\epsilon$. We compute elements of the matrix $\hat M$ up to high
  enough order of $\epsilon$ for it to be non-singular, and compute
  the lowest order term of $\kappa^2$ in terms of $\epsilon$.
  \begin{enumerate}[(a)]
    \item Here $\hat M = \frac{1}{24}\begin{pmatrix} 2 & 1 & 1 \\ 1 & 2 & 1 \\ 1 & 1 & 2 \\ \end{pmatrix}$ and $\kappa^2 = \det(\hat M^{11})/\det(\hat M) = 18$.
    \item Here $\hat M \approx \frac{\epsilon}{12}\begin{pmatrix} 4 & 2 & 3\epsilon \\ 2 & 4 & 3\epsilon \\ 3\epsilon & 3\epsilon & 4\epsilon^2 \\ \end{pmatrix}$ and $\kappa^2 = \det(\hat M^{11})/\det(\hat M) \approx 7\epsilon^{-1}$ for $\epsilon \ll 1$.
    \item Here $\hat M \approx \frac{\epsilon^2}{24}\begin{pmatrix} 12 & 4\epsilon & 4\epsilon \\ 4\epsilon & 2\epsilon^2 & \epsilon^2 \\ 4\epsilon & \epsilon^2 & 2\epsilon^2 \\ \end{pmatrix}$  and $\kappa^2 = \det(\hat M^{11})/\det(\hat M) = 18\epsilon^{-2}$ for $0 < \epsilon \le 1$.
    \item Here we can reuse $\hat M$ from (b) (since (d) is a linear
      transformation of (b) with preserved volume) and compute
      $\kappa^2 = \det(\hat M^{33})/\det(\hat M) \approx
      12\epsilon^{-3}$ for $\epsilon \ll 1$.
    \item Here $\sigma$ is a superset of $\sigma$ in (c), thus
      according to Lemma~\ref{lem:extend},
      $\kappa^2 \le 18\epsilon^{-2}$.
  \end{enumerate}
  Case (d) above can not occur with the presented definition of $\IH$,
  since $z \not \in \sigma$. Although we do not elaborate on that, the
  condition $z \in \omega^1$ and consequently $z \in \sigma$
  can be relaxed. However, it comes with the cost of larger $\kappa$.
\end{example}
\newcommand{\meshtriangle} {%
  \draw[\gridcolor] (0,0) -- (0,1) -- (1,0) -- (0,0);
}

\newcommand{\meshtrianglegreen} {%
  \fill[\greencolor] (0,0) -- (0,1) -- (1,0) -- (0,0);
}
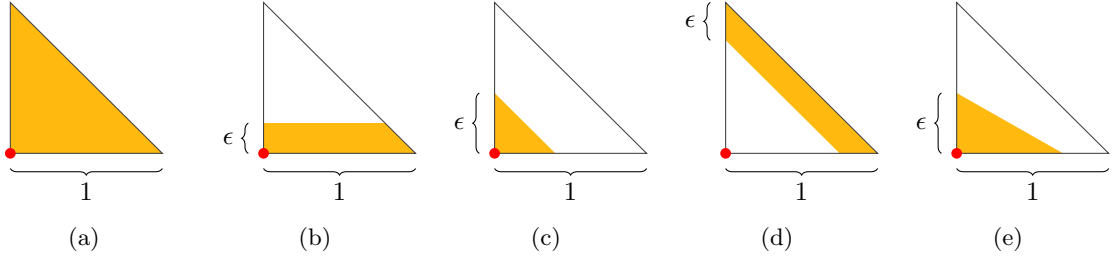
\begin{figure}[]
  \centering
  \begin{subfigure}[t]{.19\textwidth}
    \centering
    \begin{tikzpicture}[scale=2]
      \draw[decoration={brace},decorate]
      (1,-0.1) -- node[below=2pt] {$1$} (0,-0.1);

      \meshtrianglegreen;
      \meshtriangle;
      \fill[red] (0,0) circle (1pt);
    \end{tikzpicture}
    \caption{}
  \end{subfigure}%
  \begin{subfigure}[t]{.19\textwidth}
    \centering
    \begin{tikzpicture}[scale=2]
      \draw[decoration={brace},decorate]
      (1,-0.1) -- node[below=2pt] {$1$} (0,-0.1);

      \draw[decoration={brace},decorate]
      (-0.1,0) -- node[left=2pt] {$\epsilon$} (-0.1,0.2);

      \begin{scope}[scale=1]
        \clip (0, 0) rectangle (1, 0.2);
        \meshtrianglegreen;
      \end{scope}
      \meshtriangle;
      \fill[red] (0,0) circle (1pt);
    \end{tikzpicture}
    \caption{}
  \end{subfigure}%
  \begin{subfigure}[t]{.19\textwidth}
    \centering
    \begin{tikzpicture}[scale=2]
      \draw[decoration={brace},decorate]
      (1,-0.1) -- node[below=2pt] {$1$} (0,-0.1);

      \draw[decoration={brace},decorate]
      (-0.1, 0) -- node[left=2pt] {$\epsilon$} (-0.1,0.4);

      \begin{scope}[scale=0.4]
        \meshtrianglegreen;
      \end{scope}
      \meshtriangle;
      \fill[red] (0,0) circle (1pt);
    \end{tikzpicture}
    \caption{}
  \end{subfigure}%
  \begin{subfigure}[t]{.19\textwidth}
    \centering
    \begin{tikzpicture}[scale=2]
      \draw[decoration={brace},decorate]
      (1,-0.1) -- node[below=2pt] {$1$} (0,-0.1);

      \draw[decoration={brace},decorate]
      (-0.1,0.75) -- node[left=2pt] {$\epsilon$} (-0.1,1);

      \begin{scope}[scale=1]
        \clip (0, 0.75) -- (0.75, 0) -- (1, 0) -- (0, 1) -- (0, 0.75);
        \meshtrianglegreen;
      \end{scope}
      \meshtriangle;
      \fill[red] (0,0) circle (1pt);
    \end{tikzpicture}
    \caption{}
  \end{subfigure}%
  \begin{subfigure}[t]{.19\textwidth}
    \centering
    \begin{tikzpicture}[scale=2]
      \draw[decoration={brace},decorate]
      (1,-0.1) -- node[below=2pt] {$1$} (0,-0.1);

      \draw[decoration={brace},decorate]
      (-0.1, 0) -- node[left=2pt] {$\epsilon$} (-0.1,0.4);

      \begin{scope}[xscale=0.7, yscale=0.4]
        \meshtrianglegreen;
      \end{scope}
      \meshtriangle;
      \fill[red] (0,0) circle (1pt);
    \end{tikzpicture}
    \caption{}
  \end{subfigure}%
  \caption{Red dot is the node $z$. Yellow area is the shape of the
    integration domain $\sigma$.}
  \label{fig:sigmageom}
\end{figure}

Since, for class I nodes, the integration domain $\sigma$ is restricted
to be subsets of a subdomain $\omega^1$, very narrow
subomain $\omega^1$ can cause a large value of the stability constant
$\kappa$ for the node it is paired with. 
For the remainder of this paper, we define an upper bound of the
stability constants
\begin{equation*}
  \kappa = \sup_H \max_{i\,:\, z_i \in \nodesH} \kappa_i.
\end{equation*}

\subsection{Stability and approximability of $\IH$}
We present element stability and approximability results for $\IH$
based on the nodal variable stability above.
\begin{lemma}[Stability and approximability of $\IH$]
  \label{lem:stabappIH}
  Under Assumptions \ref{ass:localpoincare}, for all $v \in V$, it
  holds that
  \begin{align}
    \|v - \IH v\|_{L^{2}(T)} &\le C_{\rho,\kappa} H \|\nabla v\|_{L^2(U(T))} & \text{ for all } T\in\triH, \label{eq:IHL2approx} \\
    \|\nabla (v - \IH v)\|_{L^{2}(T)} &\le C_{\rho,\kappa} \|\nabla v\|_{L^2(U(T))} & \text{ for all } T\in\triH, \label{eq:IHH1stab}
  \end{align}
  where $C_{\rho,\kappa}$ does not depend on $H$.
\end{lemma}
\begin{proof}
  The proof follows closely the proofs in \cite{ScZh90}. Let
  $D = \{i \,:\, z_i \in \overline T \}$ be the index set of the
  vertices in $T$.
\begin{equation}
  \label{eq:IHstability}
  \begin{aligned}
    |\IH v|_{H^m(T)} & = \left|\sum_{i \in D} N_i(v) \phi_i\right|_{H^m(T)} \le C_\rho H^{-m+d/2} \sum_{i\in D} |N_i(v)| \\
    & \le C_\rho H^{-m+d/2} \sum_{i\in D} \|\psi_i\|_{L^2(\sigma_i)} \|v\|_{L^2(\sigma_i)} \stackrel{\eqref{eq:kappa}}{\le} C_{\rho} H^{-m} \max_{i\in D}\kappa_i \|v\|_{L^2(U(T))}.
  \end{aligned}
\end{equation}
To prove \eqref{eq:IHL2approx}, we use that for any $c \in \R$
\begin{equation*}
  \begin{aligned}
    \|v - \IH v\|_{L^{2}(T)} & \le \|v - c\|_{L^{2}(T)} + \|\IH(v - c)\|_{L^{2}(T)} \\
    & \stackrel{\eqref{eq:IHstability}}{\le} (1 + C_{\rho} \max_{i\in D}\kappa_i)\|v - c\|_{L^{2}(U(T))} \\
    & \le C_{\rho}(1 + \max_{i\in D}\kappa_i) H \|\nabla v \|_{L^{2}(U(T))}.\\
\end{aligned}
\end{equation*}
Bramble--Hilbert lemma was used in the last step for the element patch
$U(T)$. A similar argument is used to prove \eqref{eq:IHH1stab}.
\end{proof}

 \newcommand{\grid}[1][1] {%
   \FPupn\start{#1 1 +};
   \foreach \x in {-\start,...,#1} { \foreach \y in {-\start,...,#1} {
       \begin{scope}[shift={(\x,\y)}]
         \draw[\gridcolor] (0,0) grid (1,1);
         \draw[\gridcolor] (0,1) -- (1,0);
       \end{scope}
     }
   }
 }

\newcommand{\omegaoneshapethree} {%
  \fill[black] (-0.05, -1.7) rectangle (0.05, 1.7);
  \fill[black] (0.3, -1.7) rectangle (0.4, 1.7);
  \fill[black] (-0.4, -1.7) rectangle (-0.3, 1.7);
  \fill[black] (-0.6, 0.7) rectangle (0.6, 1.3);
  \fill[black] (-0.6, -0.3) rectangle (0.6, 0.3);
  \fill[black] (-0.6, -1.3) rectangle (0.6, -0.7);
}

\newcommand{\omegaoneshapethreered} {%
  \fill[\redcolor] (-0.05, -1.7) rectangle (0.05, 1.7);
  \fill[\redcolor] (0.3, -1.7) rectangle (0.4, 1.7);
  \fill[\redcolor] (-0.4, -1.7) rectangle (-0.3, 1.7);
  \fill[\redcolor] (-0.6, 0.7) rectangle (0.6, 1.3);
  \fill[\redcolor] (-0.6, -0.3) rectangle (0.6, 0.3);
  \fill[\redcolor] (-0.6, -1.3) rectangle (0.6, -0.7);
}

\newcommand{\omegaoneshapefour} {%
  \fill[black] (-0.05, -1.7) rectangle (0.05, 1.7);
}

\newcommand{\omegaoneshapefourred} {%
  \fill[\redcolor] (-0.05, -1.7) rectangle (0.05, 1.7);
}

\newcommand{\omegaoneshapefive} {%
  \begin{scope}[rotate=33.69]
    \fill[black] (-0.05, -1.7) rectangle (0.05, 1.7);
  \end{scope}
}

\newcommand{\omegaoneshapefivered} {%
  \begin{scope}[rotate=33.69]
    \fill[\redcolor] (-0.05, -1.7) rectangle (0.05, 1.7);
  \end{scope}
}

\begin{figure}[]
  \centering
  \begin{subfigure}[t]{.23\textwidth}
    \centering
    \begin{tikzpicture}[scale=1.5]
      \draw[decoration={brace},decorate]
      (-0.6,1.2) -- node[above=2pt] {$1$} (0.6,1.2);

      \clip (-1.1, -1.1) rectangle (1.1, 1.1);

      \omegaoneshapethree;
      \begin{scope}
        \clip (-1, -0.5) rectangle (1, 0.5);
        \omegaoneshapethreered;
      \end{scope}
      \begin{scope}
        \grid;
      \end{scope}
    \fill[red] (0,0) circle (1pt);
    \end{tikzpicture}%
    \caption{}
    \label{fig:circuit}
  \end{subfigure}%
  \begin{subfigure}[t]{.23\textwidth}
    \centering
    \begin{tikzpicture}[scale=1.5]
      \draw[decoration={brace, amplitude=+1pt},decorate]
      (-0.05,1.2) -- node[above=2pt] {$\epsilon$} (0.05,1.2);

      \clip (-1.1, -1.1) rectangle (1.1, 1.1);
      \begin{scope}[shift={(0,0.5)}]
        \omegaoneshapethree;
      \end{scope}
      \begin{scope}
        \clip (-1, -0.5) rectangle (1, 0.5);
        \begin{scope}[shift={(0,0.5)}]
          \omegaoneshapethreered;
        \end{scope}
      \end{scope}
      \begin{scope}
        \grid;
      \end{scope}
    \fill[red] (0,0) circle (1pt);
    \end{tikzpicture}%
    \caption{}
    \label{fig:circuit_displaced}
  \end{subfigure}%
  \begin{subfigure}[t]{.23\textwidth}
    \centering
    \begin{tikzpicture}[scale=1.5]
      \draw[decoration={brace, amplitude=+1pt},decorate]
      (-0.05,1.2) -- node[above=2pt] {$\epsilon$} (0.05,1.2);

      \clip (-1.1, -1.1) rectangle (1.1, 1.1);
      \begin{scope}[shift={(0,0.5)}]
        \omegaoneshapefour;
      \end{scope}
      \begin{scope}
        \clip (-1, -0.5) rectangle (1, 0.5);
        \begin{scope}[shift={(0,0.5)}]
          \omegaoneshapefourred;
        \end{scope}
      \end{scope}
      \begin{scope}
        \grid;
      \end{scope}
      \fill[red] (0,0) circle (1pt);
    \end{tikzpicture}%
    \caption{}
  \end{subfigure}%
  \begin{subfigure}[t]{.23\textwidth}
    \centering
    \begin{tikzpicture}[scale=1.5]
      \clip (-1.1, -1.1) rectangle (1.1, 1.1);
      \begin{scope}
        \omegaoneshapefive;
      \end{scope}
      \begin{scope}
        \clip (-1, -0.375) rectangle (1, 0.375);
        \begin{scope}
          \omegaoneshapefivered;
        \end{scope}
      \end{scope}
      \begin{scope}[scale=0.25]
        \grid[4];
      \end{scope}
      \fill[red] (0,0) circle (1pt);
    \end{tikzpicture}%
    \caption{}
  \end{subfigure}%
  \caption{Illustrations for Example~\ref{ex:influence}. Black
    background is $\Omega^1$. Red dot is $z_i$. Blue area is
    $\omega^1_i$.}
  \label{fig:poincarecoefficient}
\end{figure}
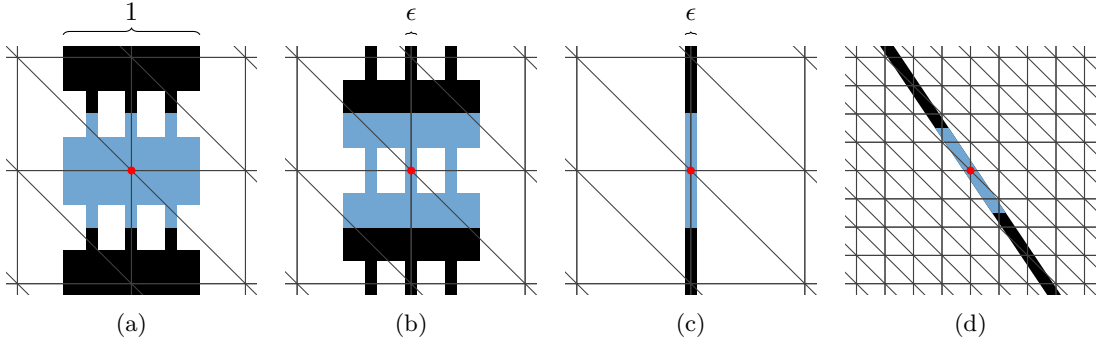
The following example illustrates how node placements affects the
stability constants $\Cdiam$, $p$ and $\kappa$.
\begin{example}[The influence of node placement on stability constants]
  \label{ex:influence}
  Consider the coefficients and node placement options (a--d) given in
  Figure~\ref{fig:poincarecoefficient}. We discuss the values of the
  constants based on the four options. We use \cite[eq. (1.1),
  (1.2)]{Ve99} and \cite[Remark 7.2]{DuSc80} to estimate upper bounds
  of the Poincar\'e constant $\Cpoinc$, and the cases in
  Example~\ref{ex:kappa} to estimate $\kappa$. We use the notation
  $A \lesssim B$ if $A \le CB$ for a constant $C$ independent of
  $\epsilon$.
  \begin{enumerate}[(a)]
  \item Here $\omega^1_i$ is a union of three star-shaped domains with
    intersection of minimum width independent of $\epsilon$. Thus
    $\Cpoinc$ can be bounded from above independent of
    $\epsilon$. Node variable stability constant $\kappa$ is benign
    (and independent of $\epsilon$) since $\sigma_i$ covers a large
    and node centered area around the node. The subdomain is within
    the node patch, why $\Cdiam = 1$.

  \item Here $\omega^1_i$ is a union of three star-shaped domains with
    intersection of minimum width $\epsilon$, hence
    $\Cpoinc \lesssim \epsilon^{-1}\log(\epsilon^{-1})^{1/2}$, which
    increases when $\epsilon$ decreases. Node variable stability
    constant $\kappa$ is larger than case (a), but can be uniformly
    bounded independent of $\epsilon$. Again, $\Cdiam = 1$.

  \item Here $\omega^1_i$ is convex and $\Cpoinc \le \pi^{-1}$, independent
    of $\epsilon$. Node variable stability is
    $\kappa \lesssim \epsilon^{-1/2}$ and $\Cdiam = 1$.

  \item Here the extent of $\omega^1_i$ is larger than the node patch, and $\Cdiam = 2$.
  \end{enumerate}
\end{example}

\section{Contrast independent error bounds for a multiscale method based on $\IH$}
\label{sec:error}
This section presents error analysis of the localized methods based on
$\IH$. Throughout this section, we fix $\IH$ as the choice of
interpolation operator in the multiscale method and thus have
$\Vf = \ker(\IH)$. The main result is that this yields error bounds
where the exponential decay rate of the localization error is
independent of the contrast $\alpha$.

The section is divided into three subsections. The first subsection
presents a few bounds for functions in $\Vf = \ker(\IH)$. These
constitute the key components for the proofs on contrast independent
localization error, given in the second subsection. The error bounds
for the full multiscale method are presented in the last and third
subsection.

\subsection{Bounds for functions in $\ker(\IH)$}
We briefly discuss what is different between the contrast independent
localization error proof in this paper and the contrast dependent
classical proof in e.g.\ \cite{MaPe14}. The classical proof make use
of inequalities of the kind
\begin{equation*}
  \|A^{1/2}\vf\|_{L^2(T)} \le \|\vf - \IH \vf\|_{L^2(T)} \le C\alpha^{-1/2} H \|A^{1/2}\nabla \vf\|_{L^2(U(T))},
\end{equation*}
where the contrast enters the bound. The results in this paper avoid
this by using that $A$ is two-valued and splitting the integral
separately over $T \cap \Omega^1$ and $T \cap \Omega^\alpha$.
The node placement within $\Omega^1$ (guaranteed by
Assumption~\ref{ass:localpoincare}) and the definition of $\IH$ yields
a Poincar\'e-type inequality over $\Omega^1$ independent of $\alpha$ for
functions in $\ker(\IH)$ as presented in \eqref{eq:IHVfapprox} in
Lemma~\ref{lem:poincareIH} below. The norm on the right hand side in
the estimate is taken over a larger domain than that on the left hand
side. However, this spreading is restricted to $\Omega^1$, which keeps
the separation of the initial integral to $\Omega^1$. For the integral
over $\Omega^\alpha$, spreading is not a problem since we can carry
the constant $\alpha^{1/2}$ to cancel the constant $\alpha^{-1/2}$
arising from the classical estimates.
\begin{lemma}[Poincar\'e-type inequalities for functions in kernel to
  $\IH$]
  \label{lem:poincareIH}
  Under Assumptions \ref{ass:localpoincare} and for all
  $\vf \in \Vf = \ker(\IH)$, it holds that
  \begin{align}
    \|\vf\|_{L^{2}(\Omega^1 \cap \omega)} &\le C_{\rho,\Cpoinc,\kappa,\Cdiam} H \|\nabla \vf\|_{L^2(\Omega^1 \cap U_\Cdiam(\omega))} & \text{ for all } \omega\subset\Omega, \label{eq:IHVfapprox} \\
    \|A^{1/2} \vf\|_{L^2(T)} & \le C_{\rho,\Cpoinc,\kappa,\Cdiam} H\|A^{1/2}\nabla \vf \|_{L^2(U_{\Cdiam}(T))} & \text{ for all } T\in\triH, \label{eq:IHVfEnergyapprox}
  \end{align}
  where $C_{\rho,\Cpoinc,\kappa,\Cdiam}$ does not depend on $\alpha$ or $H$.
\end{lemma}
\begin{proof}
For \eqref{eq:IHVfapprox}, we use that $N_i(\vf) = 0$ and
$N_i(c) = c$, for any $c \in \R$. Then, using Assumption~\ref{ass:localpoincare} and the node numbering $z_i$ described in
Section~\ref{sec:definition}, choosing $c_i \in \R$ appropriately, we
have for any $\omega \subset \Omega$,
\begin{equation*}
  \begin{aligned}
    \|\vf\|_{L^2(\Omega^1 \cap \omega)}^2 & \le \sum_{i\,:\,\omega^1_{i} \cap \omega \ne \emptyset} \|\vf\|_{L^2(\omega^1_{i})}^2 \\
    & \le 2\sum_{i\,:\,\omega^1_{i} \cap \omega \ne \emptyset} \left(\|\vf - c_i\|_{L^2(\omega^1_{i})}^2 + \|N_i(\vf - c_i)\|_{L^2(\omega^1_{i})}^2\right) \\
    & \le 2\sum_{i\,:\,\omega^1_{i} \cap \omega \ne \emptyset} \left(\|\vf - c_i\|_{L^2(\omega^1_{i})}^2 + C_\rho \Cdiam^{d} \kappa_i^2 \|\vf - c_i\|_{L^2(\sigma_i)}^2\right) \\
    & \stackrel{\eqref{eq:poincare}}{\le} C_\rho \sum_{i\,:\,\omega^1_{i} \cap \omega \ne \emptyset} (1 + \Cdiam^{d} \kappa_i^2 ) \Cpoinc^2 H^2 \|\nabla \vf\|_{L^2(\omega^1_{i})}^2 \\
    & \le C_{\rho, \Cpoinc, \kappa, \Cdiam} H^2 \|\nabla \vf\|_{L^2(\Omega^1 \cap U_\Cdiam(\omega))}^2. \\
  \end{aligned}
\end{equation*}

For \eqref{eq:IHVfEnergyapprox}, we combine the two obtained
results \eqref{eq:IHL2approx} and \eqref{eq:IHVfapprox} to get
  \begin{equation*}
    \begin{aligned}
      \|A^{1/2} \vf\|_{L^2(T)} & \le \alpha^{1/2} \|\vf\|_{L^2(\Omega^\alpha \cap T)} + \|\vf\|_{L^2(\Omega^1 \cap T)} \\
      & \stackrel{\eqref{eq:IHVfapprox}}{\le} \alpha^{1/2} \|\vf -\IH \vf \|_{L^2(T)} + C_{\rho,\Cpoinc,\kappa,\Cdiam}H\|\nabla \vf \|_{L^2(\Omega^1 \cap U_\Cdiam(T))}  \\
      & \stackrel{\eqref{eq:IHL2approx}}{\le} \alpha^{1/2} C_{\rho,\kappa} H \|\nabla \vf \|_{L^2(U(T))} + C_{\rho,\Cpoinc,\kappa,\Cdiam} H \|\nabla \vf \|_{L^2(\Omega^1 \cap U_\Cdiam(T))}\\
      & \le C_{\rho,\Cpoinc,\kappa,\Cdiam}H\|A^{1/2}\nabla \vf \|_{L^2(U_{\Cdiam}(T))}. \\
    \end{aligned}
  \end{equation*}
\end{proof}

We define a slightly different cut-off function compared to previous
works, which is constant on $\Sigma_\delta(z_i)$ for all nodes
$z_i$. For the remainder of the paper we will set $\delta=1/4$
as in the definition of $\IH$ and drop the $\delta$-subscript. The
operator $\IHinf$ for which $\delta = 1$ will not be analyzed and
appears again in the numerical experiment section. We introduce, for
$T \in \triH$ and $k > 0$, the auxiliary function
${\hat \eta}_{T,k} \in S_H$,
\begin{equation}\label{e:cutoffH}
\begin{aligned}
 {\hat \eta}_{T,k}(x) &= 0\quad\text{for } x \in U_{k-1}(T),\\
 {\hat \eta}_{T,k}(x) &= 1\quad\text{for } x \in \Omega\setminus U_k(T),\\
\end{aligned}
\end{equation}
and define the cut-off function as
$\eta_{T,k} = \max(0, \min(1, 2{\hat \eta}_{T,k} - 1/2))$.
We get 
$\|\nabla \eta_{T,k}\|_{L^{\infty}(\Omega)} \le 2C_\rho H^{-1}$.
Figure~\ref{fig:cutoff} shows an illustration of such a cut-off
function.
\begin{figure}[]
  \centering
  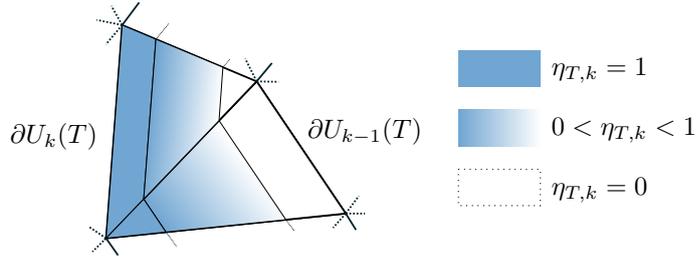
  \caption{A 2D cut-off function illustrated.}
  \label{fig:cutoff}
\end{figure}

With the definition of the cut-off functions at hand, we can study a
second term which appears in the classical proof and where contrast
enters the bound, namely
\begin{equation*}
  \|A^{1/2}\nabla \IH(\eta \vf)\|_{L^2(T)} \le C \|\nabla (\eta \vf)\|_{L^2(U(T))} \le C\alpha^{-1/2} \|A^{1/2}\nabla \vf\|_{L^2(U_2(T))},
\end{equation*}
for a cut-off function $\eta$. With $\IH$, this term can also bounded
independently of $\alpha$ using similar techniques as in the previous lemma.
\begin{lemma}[Local contrast independent energy norm stability of
  \fine{} space functions after cut-off and interpolation]
  \label{lem:aux1}
  If $\eta \in L^{\infty}(\Omega)$ is constant in $\Sigma(z_i)$
  for all nodes $z_i$ and $\|\eta\|_{L^{\infty}(\Omega)} \le 1$, then
  for any $\vf \in \Vf$ and $\omega \subset \Omega$,
  \begin{equation}
    \label{eq:aux1}
    \|A^{1/2} \nabla (\IH(\eta \vf))\|_{L^2(\omega)} \le C_{\rho, \Cpoinc, \kappa, \Cdiam}\|A^{1/2} \nabla \vf \|_{L^2(U_{\Cdiam+2}(\omega))},
  \end{equation}
  where $C_{\rho, \Cpoinc, \kappa, \Cdiam}$ does not depend on $\alpha$ or $H$.
\end{lemma}
\begin{proof}
  We split the integral into one contribution over $\Omega^1$ and one
  over $\Omega^\alpha$. For $\Omega^1$, 
  \begin{equation*}
    \begin{aligned}
    \|\nabla (\IH(\eta \vf))\|^2_{L^2(\Omega^1 \cap \omega)} & = \int_{\Omega^1 \cap \omega} \left(\sum_{i\,:\,z_i \in \nodesH} N_i(\eta \vf) \nabla \phi_i\right)^2 \\
    & \le C_\rho \sum_{i\,:\,z_i \in \nodesH} \left( \int_{\sigma_i} \psi_i \eta \vf \cdot \int_{\Omega^1 \cap \omega} \nabla \phi_i\right)^2 \cdots \\
    \end{aligned}
  \end{equation*}
  We study the two integrals. The following step justifies the
  partition into the two classes introduced in
  Section~\ref{sec:definition}.  For any node $z_i \in \nodesIIH$,
  $\eta$ is constant in $\sigma_i$ and
  $\int_{\sigma_i} \psi_i \eta \vf = \eta \int_{\sigma_i} \psi_i \vf =
  0$
  since $\vf \in \Vf$.
  This makes the corresponding term in the sum is zero. We define the set
  $\mathcal{N}^1_{\omega} = \{z_i \,:\, z_i \in \nodesIH \text{ and }
  U(z_i)\cap\Omega^1\cap\omega \ne \emptyset\}$,
  which contains all nodes whose corresponding term in the sum is not
  zero in general, and continue using mesh quasi-uniformity,
  Lemma~\ref{lem:psil2}, that $\|\eta\|_{L^{\infty}(\Omega)} \le 1$,
  and $\sigma_i \subset U(z_i)\cap\Omega^1$ for
  $z_i \in \mathcal{N}^1_{\omega}$,
  \begin{equation*}
    \begin{aligned}
    \cdots & \le C_\rho H^{-2-d} \sum_{i\,:\, z_i \in \mathcal{N}^1_{\omega}} \left(\int_{\sigma_i} \psi_i \eta \vf \right)^2 \\
    & \stackrel{\eqref{eq:kappa}}{\le} C_{\rho,\kappa} H^{-2} \sum_{i\,:\, z_i \in \mathcal{N}^1_{\omega}} \|\vf\|^2_{L^2(\sigma_i)} \\
    & \le C_{\rho,\kappa}  H^{-2} \sum_{i\,:\, z_i \in \mathcal{N}^1_{\omega}} \|\vf\|^2_{L^2(\Omega^1 \cap U(z_i))} \\
    & \stackrel{\eqref{eq:IHVfapprox}}{\le} C_{\rho,\Cpoinc,\kappa,\Cdiam}  \sum_{i\,:\, z_i \in \mathcal{N}^1_{\omega}} \|\nabla \vf\|^2_{L^2(\Omega^1 \cap U_{\Cdiam+1}(z_i))} \\
    & \le C_{\rho,\Cpoinc,\kappa,\Cdiam} \|\nabla \vf\|^2_{L^2(\Omega^1 \cap U_{\Cdiam+2}(\omega))} \\
    \end{aligned}
  \end{equation*}
  where we used shape regularity and that
    $\bigcup_{i\,:\, z_i \in \mathcal{N}^1_{\omega}} U_{\Cdiam+1}(z_i)
    \subset U_{\Cdiam+2}(\omega)$
  in the last step.
  For $\Omega^\alpha$, we carry the constant $\alpha \le A$ through
  the steps and ignore spreading to $\Omega^1$,
  \begin{equation*}
    \begin{aligned}
      \alpha^{1/2} \|\nabla (\IH(\eta \vf))\|_{L^2(\Omega^\alpha \cap \omega)} & \stackrel{\eqref{eq:IHH1stab}}{\le} \alpha^{1/2} C_{\rho,\kappa} \|\nabla (\eta \vf)\|_{L^2(U_1(\omega))} \\
      & \le C_{\rho,\kappa}H^{-1}\|A^{1/2} \vf\|_{L^{2}(U_1(\omega))} + C_{\rho,\kappa} \| A^{1/2} \nabla \vf \|_{L^{2}(U_1(\omega))} \\
      & \stackrel{\eqref{eq:IHVfEnergyapprox}}{\le} C_{\rho,\Cpoinc,\kappa,\Cdiam} \|A^{1/2} \nabla \vf\|_{L^{2}(U_{\Cdiam+1}(\omega))}.
    \end{aligned}
  \end{equation*}
  The assertion follows from adding the two contributions.
 \end{proof}

\subsection{Contrast independent localization error bounds}
The proof technique used here follows closely e.g.\ \cite{HMP13b}.
The first lemma shows that non-local correctors exhibit exponential
decay and the second shows that the localization error decays
exponentially with increasing patch size $k$ independent of
$\alpha$. We use the following elementwise decomposition of
$Qv = \sum_{T \in \triH} Q_T v$ and $Rf = \sum_{T \in \triH} R_Tf$,
where, for all $\vf \in \Vf$,
\begin{equation*}
a(Q_Tv, \vf) = \int_T A \nabla v \cdot \nabla \vf \qquad\text{ and }\qquad a(R_Tv, \vf) = \int_T f \vf.
\end{equation*}
With this decomposition, the lemmas below can be applied with
$p^T = Q_Tv$ or $p^T = R_Tv$ for any $T \in \triH$ and $v \in V$.
\begin{lemma}[Contrast independent exponential decay of non-local corrector]
  \label{lem:decay}
  Let $F_T(v) \in V'$ with $F_T(v) = 0$ for $v \in V(\Omega \setminus T)$. Let
  $p^T \in \Vf$ satisfy, for all $v \in \Vf$,
  \begin{equation*}
    \int_\Omega A \nabla p^T \cdot \nabla v = F_T(v).
  \end{equation*}
  Under Assumptions \ref{ass:localpoincare}, it holds for $k \ge \Cdiam+5$,
  \begin{equation}
    \label{eq:decay}
    \|A^{1/2} \nabla p^T \|_{L^2(\Omega \setminus U_k(T))} \le C_{\rho, \Cpoinc, \kappa, \Cdiam}\theta^{k} \| A^{1/2} \nabla p^T \|_{L^2(\Omega)},
  \end{equation}
  where $0 < \theta < 1$ and $C_{\rho, \Cpoinc, \kappa, \Cdiam}$ do not depend on $\alpha$, $k$ or $H.$
\end{lemma}
\begin{proof}
  \newcommand{\ktmp}{m}
  Let $\ktmp = k-\Cdiam-3 \ge 2$. 
  For brevity, we drop a few indices and denote by $U_\ktmp  = U_\ktmp (T)$,
  $p = p^T$ and $\eta = \eta^T_\ktmp $.
  \begin{equation*}
    \begin{aligned}
      \|A^{1/2} \nabla p\|^2_{L^2(\Omega \setminus U_\ktmp)} & = \int_{\Omega \setminus U_\ktmp } A \nabla p \cdot \nabla p 
      \le \int_{\Omega \setminus U_{\ktmp -1}} A \nabla p \cdot \eta \nabla p \\
      & = \underbrace{\int_{\Omega \setminus U_{\ktmp -2}} A \nabla p \cdot \nabla (\eta p)}_{=:\rm I} - \underbrace{\int_{\Omega \setminus U_{\ktmp -1}} A \nabla p \cdot p \nabla \eta}_{=:\rm II}.
    \end{aligned}
  \end{equation*}
  For term ${\rm I}$, we use that
  $\int_{\Omega\setminus U_{\ktmp -2}} A\nabla p \cdot \nabla(\eta p
  -\IH(\eta p)) = \int_{\Omega}A\nabla p \cdot \nabla(\eta p -\IH(\eta
  p)) = F_T(\eta p -\IH(\eta p)) = 0$,
  since $\eta p -\IH(\eta p)$ has no support in $U_{m-2}(T) \supset T$ (requires
  $\ktmp  \ge 2$). Since
  $\supp(\IH(\eta p)) \subset U_{\ktmp +1} \setminus U_{\ktmp -2}$ we get
  \begin{equation*}
    \begin{aligned}
      |{\rm I}| & = \left|\int_{U_{\ktmp +1} \setminus U_{\ktmp -2}} A \nabla p \cdot
      \nabla (\IH(\eta p)) \right| \\ 
      & \stackrel{\eqref{eq:aux1}}{\le} C_{\rho,\Cpoinc,\kappa,\Cdiam} \|A^{1/2} \nabla p\|_{L^2(U_{\ktmp +1} \setminus U_{\ktmp -2})} \|A^{1/2} \nabla p\|_{L^2(U_{\ktmp +\Cdiam+3} \setminus U_{\ktmp -\Cdiam-4})}.
    \end{aligned}
  \end{equation*}
  For term $\rm II$, we use that
  $\supp(\nabla \eta) \subset U_{\ktmp } \setminus U_{\ktmp -1}$ and
  $\|\nabla \eta \|_{L^{\infty}(\Omega)} \le C_{\rho} H^{-1}$ to get
  \begin{equation*}
    \begin{aligned}
      |{\rm II}| & \le C_{\rho} H^{-1} \|A^{1/2} \nabla p\|_{L^2(U_{\ktmp } \setminus U_{\ktmp -1})} \|A^{1/2} p\|_{L^2(U_{\ktmp } \setminus U_{\ktmp -1})} \\
      & \stackrel{\eqref{eq:IHVfEnergyapprox}}{\le} C_{\rho, \Cpoinc, \kappa, \Cdiam} \|A^{1/2} \nabla p\|_{L^2(U_{\ktmp } \setminus U_{\ktmp -1})} \|A^{1/2}\nabla p \|_{L^2(U_{\ktmp +\Cdiam} \setminus U_{\ktmp -\Cdiam-1})}.
    \end{aligned}
  \end{equation*}
  We obtain, with $c=C_{\rho, \Cpoinc, \kappa, \Cdiam}$,
  \begin{equation*}
    \begin{aligned}
      \|A^{1/2} \nabla p\|^2_{L^2(\Omega \setminus U_{\ktmp +\Cdiam+3})} & \le \|A^{1/2} \nabla p\|^2_{L^2(\Omega \setminus U_\ktmp)} \le c\| A^{1/2} \nabla p \|^2_{L^{2}(U_{\ktmp +\Cdiam+3} \setminus U_{\ktmp -\Cdiam-4})} \\
      & = c(\| A^{1/2} \nabla p \|^2_{L^{2}(\Omega \setminus U_{\ktmp -\Cdiam-4})} - \| A^{1/2} \nabla p \|^2_{L^{2}(\Omega \setminus U_{\ktmp +\Cdiam+3})}).
    \end{aligned}
  \end{equation*}
  from which we get the following decay result over $2\Cdiam+7$ patch layers,
  \begin{equation}
    \label{eq:telescope}
    \begin{aligned}
      \|A^{1/2} \nabla p\|^2_{L^2(\Omega \setminus U_{\ktmp +\Cdiam+3})} \le \frac{c}{1+c} \|A^{1/2} \nabla p\|^2_{L^2(\Omega \setminus U_{\ktmp -\Cdiam-4})}.
    \end{aligned}
  \end{equation}
  Successive application of \eqref{eq:telescope} yields
  \begin{equation*}
  \|A^{1/2} \nabla p\|^2_{L^2(\Omega \setminus U_{\ktmp +\Cdiam+3})} \le \left(\frac{c}{1+c}\right)^{\lfloor \ktmp /(2\Cdiam+7) \rfloor} \|A^{1/2} \nabla p\|^2_{L^2(\Omega)}.
\end{equation*}
  The final result is obtained by choosing
  $\theta = \left(\frac{c}{1+c}\right)^{1/(4\Cdiam+14)}$, which is
  independent of $\alpha$.
\end{proof}

\begin{lemma}[Contrast independent localization error]
  \label{lem:localization_error}
  Let $F_T(v) \in V'$ with $F_T(v) = 0$ for
  $v \in V(\Omega \setminus T)$.  Let $p^T \in \Vf$ satisfy, for all
  $v \in \Vf$,
  \begin{equation*}
    \int_\Omega A \nabla p^T \cdot \nabla v = F_T(v),
  \end{equation*}
  and $p = \sum_{T\in\triH} p^T$. Further, let $p_k^T \in \Vf(U_k(T))$
  satisfy, for all $v \in \Vf(U_k(T))$,
  \begin{equation*}
    \int_{U_k(T)} A \nabla p_k^T \cdot \nabla v = F_T(v),
  \end{equation*}
  and $p_k = \sum_{T\in\triH} p_k^T$. Then, for $k \ge \Cdiam+5$,
  \begin{equation*}
    \|A^{1/2} \nabla (p -p_k) \|_{L^2(\Omega)} \le C_{\rho, \Cpoinc, \kappa, \Cdiam} k^{d/2}\theta^{k} \left(\sum_{T \in \triH} \| A^{1/2} \nabla p^T \|^2_{L^2(\Omega)}\right)^{1/2},
  \end{equation*}
  where $0 < \theta < 1$ and
  $C_{\rho, \Cpoinc, \kappa, \Cdiam}$ do not depend on $\alpha$, $k$ or $H$.
\end{lemma}
\begin{proof}
  Let $z = p - p_k = \sum_{T\in\triH} (p^T - p_k^T)$.  Since, for all
  $T$, $\tilde z = \eta^T_{k+2}z - \IH(\eta^T_{k+2}z) \in \Vf$, with
  $F_T(\tilde z) = 0$ and $p^T_k$ and $\tilde z$ lack common support,
  we have $\int_\Omega A\nabla \tilde z \cdot \nabla (p^T - p^T_k) = 0$ and we get
  \begin{equation}
    \label{eq:localmain}
    \begin{aligned}
      \|A^{1/2} \nabla z\|^2_{L^2(\Omega)} & = \sum_{T \in \triH} \int_{\Omega} A \nabla ((1 - \eta^T_{k+2})z + \IH(\eta^T_{k+2}z)) \cdot \nabla (p^T - p_k^T) \\
      & \le \sum_{T \in \triH} \Big(\underbrace{\| A^{1/2} \nabla ((1 - \eta^T_{k+2})z)\|_{L^2(\Omega)}}_{=:\rm I}  + \underbrace{\| A^{1/2} \nabla \IH(\eta^T_{k+2}z) \|_{L^2(\Omega)}}_{=:\rm II} \Big) \cdot \\
      & \phantom{\le{}} \cdot \underbrace{\| A^{1/2} \nabla (p^T - p_k^T) \|_{L^2(\Omega)}}_{=:\rm III}. \\
    \end{aligned}
  \end{equation}
  For term $\rm I$, we have
  \begin{equation*}
    \begin{aligned}
      \|A^{1/2} \nabla ((1 - \eta^T_{k+2})z) \|_{L^{2}(\Omega)} & \le \|A^{1/2} \nabla ((1 - \eta^T_{k+2})z) \|_{L^{2}(U_{k+2}(T))} \\
        & \le C_{\rho} H^{-1} \| A^{1/2} z\|_{L^2(U_{k+2}(T) \setminus U_{k+1}(T))} + \| A^{1/2} \nabla z\|_{L^2(U_{k+2}(T))} \\
        & \stackrel{\eqref{eq:IHVfEnergyapprox}}{\le}  C_{\rho, \Cpoinc, \kappa, \Cdiam} \| A^{1/2} \nabla z\|_{L^2(U_{k+\Cdiam+2}(T))}.
    \end{aligned}
  \end{equation*}
  For term $\rm II$, we have
  \begin{equation*}
    \begin{aligned}
      \| A^{1/2} \nabla \IH(\eta^T_{k+2}z) \|_{L^2(\Omega)} & = \| A^{1/2} \nabla \IH(\eta^T_{k+2}z) \|_{L^2(U_{k+3}(T) \setminus U_{k+1}(T))} \\
      & \stackrel{\eqref{eq:aux1}}{\le} C_{\rho,\Cpoinc,\kappa,\Cdiam} \|A^{1/2} \nabla z\|_{L^2(U_{k+\Cdiam+5}(T))}.
    \end{aligned}
  \end{equation*}
  For factor $\rm III$, we note that
  $\tilde p^T_k = (1-\eta^T_{k-1})p^T - \IH((1-\eta^T_{k-1})p^T) \in
  \Vf(U_k(T))$
  and that $p^T_k$ is best-approximation of $p^T$ in $\Vf(U_k(T))$ in energy norm,
  why
  \begin{equation*}
    \begin{aligned}
      \|A^{1/2}\nabla(p^T - p_k^T)\|_{L^2(\Omega)} & \le \|A^{1/2}\nabla(p^T - \tilde p^T_k)\|_{L^2(\Omega)} \\
      & \le \|A^{1/2}\nabla(\eta^T_{k-1}p^T - \IH(\eta^T_{k-1} p^T))\|_{L^2(\Omega)} \\
      & \stackrel{\eqref{eq:IHVfEnergyapprox}}{\le} C_{\rho,\Cpoinc,\kappa,\Cdiam} \|A^{1/2}\nabla p^T\|_{L^2(\Omega \setminus U_{k-\Cdiam-2}(T))} + {}\\
      & \phantom{\le{}}+\|A^{1/2} \nabla \IH(\eta^T_{k-1} p^T))\|_{L^2(U_{k}(T) \setminus U_{k-3}(T)))} \\
      & \stackrel{\eqref{eq:aux1}}{\le} C_{\rho,\Cpoinc,\kappa,\Cdiam} \|A^{1/2} \nabla p^T\|_{L^2(\Omega \setminus U_{k-\Cdiam-3}(T))} \\
      & \stackrel{\eqref{eq:decay}}{\le} C_{\rho,\Cpoinc,\kappa,\Cdiam} \theta^{k-\Cdiam-3}\|A^{1/2}\nabla p^T\|_{L^2(\Omega)}. \\
    \end{aligned}
  \end{equation*}
  Collecting the terms and, we continue from equation \eqref{eq:localmain} using H\"older's inequality and obtain
  \begin{equation*}
    \begin{aligned}
      \|A^{1/2} \nabla z\|^2_{L^2(\Omega)}  & \le C_{\rho,\Cpoinc,\kappa,\Cdiam} \theta^{k-\Cdiam-3} \sum_{T \in \triH} \|A^{1/2}\nabla z\|_{L^2(U_{k+\Cdiam+5}(T))} \|A^{1/2}\nabla p^T\|_{L^2(\Omega)}\\
      & \le  C_{\rho,\Cpoinc,\kappa,\Cdiam} (k+\Cdiam+5)^{d/2} \theta^{k-\Cdiam-3} \|A^{1/2}\nabla z\|_{L^2(\Omega)} \left(\sum_{T \in \triH} \|A^{1/2}\nabla p^T\|^2_{L^2(\Omega)}\right)^{1/2}.\\
    \end{aligned}
  \end{equation*}
  The lemma follows, since $\Cdiam + 5 \le k$.
\end{proof}

\subsection{Error bounds for localized multiscale method}
As a result of the contrast independent localization error, we get the
following error bounds for the localized multiscale method, without
and with right hand side correction.
\begin{theorem}[Error bound without right hand side correction]
  \label{lem:full_error_norhs}
  If $\umsHk$ is computed as described in Section~\ref{sec:localized}, then
  \begin{equation*}
    \begin{aligned}
      \III{u - \umsHk} \le C_{\Omega, \rho, \Cpoinc, \kappa, \Cdiam} \left(\alpha^{-1} k^{d/2}\theta^{k} \|f\|_{L^2(\Omega)} + H \|A^{-1/2}f\|_{L^{2}(\Omega)}\right),
    \end{aligned}
  \end{equation*}
  where $C_{\Omega, \rho, \Cpoinc, \kappa, \Cdiam}$ and $0 < \theta < 1$ are
  independent of $\alpha$, $k$, and $H$.
\end{theorem}
\begin{proof}
  Galerkin orthogonality implies
  $\III{u - \umsHk} \le \III{u - \vmsHk}$ for any
  $\vmsHk \in \VmsHk$. We decompose
  $u = \umsH + \uf = u_H - Qu_H + \uf$ and choose
  $\vmsHk = u_H - Q_k u_H$. Then by Lemma~\ref{lem:localization_error},
  \begin{equation*}
    \begin{aligned}
      \III{u - \umsHk} & \le \III{\umsH - \vmsHk + \uf} \le \III{Qu_H - Q_ku_H} + \III{\uf} \\
      & \le C_{\rho, \Cpoinc, \kappa, \Cdiam} k^{d/2}\theta^{k} \left(\sum_{T \in \triH} \III{Q_T u_H}^2\right)^{1/2} + \III{\uf}. \\
    \end{aligned}
  \end{equation*}
  For the first term, using \eqref{eq:IHH1stab}, we get
  \begin{equation*}
    \begin{aligned}
      \left(\sum_{T \in \triH} \III{Q_T u_H}^2\right)^{1/2} & \le \left(\sum_{T \in \triH} \|A^{1/2} \nabla u_H \|^2_{L^2(T)} \right)^{1/2} = \III{u_H} 
      = \III{\IH(\umsH)} \le \|\nabla \IH(\umsH)\|_{L^2(\Omega)} \\
      &\le C_{\rho,\kappa} \|\nabla \umsH\|_{L^2(\Omega)} \le \alpha^{-1/2} C_{\rho,\kappa} \III{\umsH} \le \alpha^{-1} C_{\Omega,\rho,\kappa} \|f\|_{L^2(\Omega)}.
    \end{aligned}
  \end{equation*}
  For the second term, we get
  \begin{equation*}
    \begin{aligned}
      \III{\uf}^2 & = (A^{-1/2}f,A^{1/2}\uf) \le \| A^{-1/2} f\|_{L^2(\Omega)} \| A^{1/2} \uf \|_{L^2(\Omega)}\\
      & \le \| A^{-1/2} f\|_{L^2(\Omega)} (\alpha^{1/2}\| \uf - I_H \uf\|_{L^2(\Omega^\alpha)} + \|\uf \|_{L^2(\Omega^1)}) \\
      & \stackrel{\eqref{eq:IHL2approx}, \eqref{eq:IHVfapprox}}{\le} C_{\rho,\Cpoinc,\kappa,\Cdiam} H \| A^{-1/2} f\|_{L^2(\Omega)} \| A^{1/2}\nabla \uf \|_{L^2(\Omega)} \\
      &= C_{\rho,\Cpoinc,\kappa,\Cdiam} H \| A^{-1/2} f\|_{L^2(\Omega)} \III{\uf}, \\
    \end{aligned}
  \end{equation*}
  which proves the lemma.
\end{proof}
\begin{theorem}[Error bound with right hand side correction]
  \label{lem:full_error}
  If $u_{H,k} = \umsLHk + \ufk$ is computed as described in Section~\ref{sec:rhscorrection}, then
  \begin{equation*}
      \III{u - u_{H,k}} \le C_{\Omega, \rho, \Cpoinc, \kappa, \Cdiam} k^{d/2}\theta^{k} \left(\alpha^{-1} \|f\|_{L^2(\Omega)} + H \|A^{-1/2} f\|_{L^2(\Omega)} \right),
  \end{equation*}
  where $C_{\Omega, \rho, \Cpoinc, \kappa, \Cdiam}$ and
  $0 < \theta < 1$ are independent of $\alpha$, $k$, and $H$. The
  second term is dominated by the first, but kept for easy comparison
  with the previous lemma.
\end{theorem}
\begin{proof}
  Galerkin orthogonality implies
  $\III{\umsH - \umsLHk} \le \III{\umsH - \vmsHk}$ for any
  $\vmsHk \in \VmsHk$. We decompose
  $u = \umsH + \uf = u_H - Qu_H + \uf$ and choose
  $\vmsHk = u_H - Q_k u_H$. Then by Lemma~\ref{lem:localization_error},
  \begin{equation*}
    \begin{aligned}
      \III{u - u_{H,k}} & \le \III{\umsH - \umsLHk} + \III{\uf - \ufk} \le \III{Qu_H - Q_ku_H} + \III{Rf - R_kf} \\
      & \le C_{\rho, \Cpoinc, \kappa, \Cdiam} k^{d/2}\theta^{k} \left(\left(\sum_{T \in \triH} \III{Q_T u_H}^2\right)^{1/2} + \left(\sum_{T \in \triH} \III{R_T f}^2\right)^{1/2}\right). \\
    \end{aligned}
  \end{equation*}
  The first term is bounded as in
  Theorem~\ref{lem:full_error_norhs}. For the second term, we use
  \eqref{eq:IHVfEnergyapprox} in Lemma~\ref{lem:poincareIH} and get
  \begin{equation*}
    \begin{aligned}
    \III{R_Tf}^2 & \le \|A^{-1/2} f\|_{L^2(T)} \|A^{1/2} R_T f\|_{L^2(T)} \le C_{\rho,\Cpoinc,\kappa,\Cdiam} H \|A^{-1/2} f\|_{L^2(T)} \III{R_T f},
    \end{aligned}
  \end{equation*}
  which concludes the proof.
\end{proof}

\section{Numerical experiments} 
\label{sec:num}
We present three numerical experiments with different coefficients to
illustrate how the full error of the solution with right hand side
correction (Theorem~\ref{lem:full_error}) depends on the choice of
patch size $k$ and contrast $\alpha$ for six different interpolation
operators. The results are discussed in Section~\ref{sec:discussion}.

\subsection{Mesh and coefficient}
The computational domain is $\Omega = [0,1]^2$ in all numerical
experiments and we use the mesh family shown in
Figure~\ref{fig:mesh}. In addition to the coarse mesh $\triH$
parametrized by $H$, we use a fine mesh $\trih$ parametrized by a fine
mesh parameter $h < H$.  Although the method above is described in
terms an infinite dimensional full space $V$, we use the fine mesh to
discretize $V$ into a $\mathcal{P}1$ FE space $V_h$, with basis
functions $\phi_{h,x}$ for nodes $x$ and $h = 2^{-10}$.  The
coefficient $A$ is defined as a piecewise constant function on the
fine mesh, taking either value $1$ or $\alpha$ on every fine element.
In all figures depicting a coefficient, black corresponds to $\Omega^1$
and white to $\Omega^\alpha$.
\begin{figure}[h]
  \centering
  \begin{subfigure}{.4\textwidth}
    \centering
    \begin{tikzpicture}[scale=1.5]
      \begin{scope}
        \gridone;
      \end{scope}
    \end{tikzpicture}%
    \caption{Coarsest mesh, $H = 1$.}
  \end{subfigure}
  \hspace{1em}
  \begin{subfigure}{.4\textwidth}
    \centering
    \begin{tikzpicture}[scale=1.5]
      \begin{scope}[scale=0.5]
        \gridtwo;
      \end{scope}
    \end{tikzpicture}%
    \caption{One refinement, $H = 1/2$.}
  \end{subfigure}
  \caption{Family $\triH$ of triangulations of the unit square.}
  \label{fig:mesh}
\end{figure}%

\subsection{Interpolation operators}
The choice of interpolation operator $\IIH$ determines the properties
of the multiscale method. Besides $\IH$ and $\IHinf$, we study four
additional interpolation operators. We comment on all of them below.

\subsubsection{Operators $\IH$ and $\IHinf$}
The formal definitions of $\IH$ and $\IHinf$ can be found in
Section~\ref{sec:definition} and are based on the subdomains
$\omega^1_i$ through the integration domains $\sigma_i$. The procedure
is reversed in our implementation: $\omega^1_i$ is implicitly defined
from $\sigma_i$, which is determined in the following procedure. For
each node $z_i$:
\begin{enumerate}
\item If $A = \alpha$ in all fine elements neighboring $z_i$ it is a class II
  node and $\sigma_i = \Sigma_\delta(z_i)$, i.e.\ 1/4 scaling of a
  node-patch for $\IH$, and full node-patch for $\IHinf$.
\item Otherwise the node is a class I node. Select any element neighboring
  $z_i$ with $A = 1$. Let $\sigma_i$ be all fine elements within the
  node patch for which there is an edge-incident path of elements with
  $A=1$ to the selected element. This guarantees that $\sigma_i$ is
  always a connected subset of $\Omega^1$.
\end{enumerate}
As a result of this procedure, whether
Assumption~\ref{ass:localpoincare} is satisfied or not depends on the
node placement.

\subsubsection{Node patch Scott--Zhang $\IHsz$}
\label{sec:IHsz}
The node patch Scott--Zhang quasi-interpolation is denoted by
$\IHsz : V \to \VH$ and is defined, for every free node $z_i \in \nodesfreeH$, by
\begin{equation*}
  (\IHsz v)(z_i) = N_i(v)
\end{equation*}
with $\sigma_i = U(z_i)$. Here, $N_i$ is defined as in
Section~\ref{sec:IHsz}.  This operator was introduced in
\cite{ScZh90}, where also the stability and approximability properties
asserted in Assumption~\ref{ass:qio} were shown. This operator is a
Cl\'ement-type quasi-interpolation operators similar to those used in
many previous works on the LOD technique, e.g.\ \cite{AbHe16, BrPe16,
  ElGiHe14, GaPe15, HeMa14, OhVe16, Pe15}.

\subsubsection{Nodal interpolation $\IHnodal$}
We denote the nodal interpolation operator by
$\IHnodal : C^0(\Omega) \to \VH$ and it is defined by
\begin{equation*}
(\IHnodal v)(z) = v(z)
\end{equation*}
for all free nodes $z \in \nodesfreeH$. Note that this operator is not
defined for all functions in $V$ for $d \ge 2$, and hence does not
fulfill Assumption~\ref{ass:qio}. Nevertheless, it is of interest to
include, since it is easy to implement and it is well-defined if $V$
is a finite element space on a very fine mesh. This is the case in all
numerical experiments performed in this paper.

\subsubsection{A-weighted projection $\IHproj$}
The A-weighted projective quasi-interpolation operator is
denoted by $\IHproj : V \to \VH$ and is defined by, for all free nodes
$z \in \nodesfreeH$,
\begin{equation*}
  (\IHproj v)(z) = (P_zv)(z),
\end{equation*}
where $P_z$ is the projection $P_zv \in S_H|_{U(z)}$ (functions in
$S_H$ restricted to node patch $U(z)$) such that
$\int_{U(z)} AP_zvw = \int_{U(z)} Avw$ for all $w \in S_H|_{U(z)}$.
This operator was used in \cite{PeSc16} to define the \fine{} space
and to show contrast independent localization error bounds under
quasi-monotonicity assumptions on the coefficient distribution within
the node patches. This assumption can be restrictive for coarse
meshes.

\subsubsection{A-weighted projection with quasi-monotonicity $\IHprojs$}
The A-weighted projective quasi-interpolation operator with guaranteed
quasi-monotonicity is denoted by $\IHprojs : V \to \VH$ and is defined
like $\IHproj$, but with the node patch $U(z)$ replaced by a connected
subset of $U(z)$ that contains $z$ and has (type-$(d-1)$)
quasi-monotone coefficient distribution. See e.g.\ \cite{PeSc16} for
the definition of quasi-monotone coefficient distribution.

\subsection{Experiment: Full error versus patch size and contrast}
\label{sec:num_localization}
Below, we present three different coefficient geometries $\Omega^1$
and study the full error between a reference solution on the finest
grid and the solution from the localized multiscale method with right
hand side correction while varying interpolation operator, patch size
and contrast. In all experiments in this subsection, we let
$H = 2^{-4}$ and solve the right hand side corrected solution
$u_{H,k}$ for
$\IIH \in \{\IHsz, \IHnodal, \IH, \IHinf, \IHproj, \IHprojs\}$,
$k = 1,\ldots,6$, and $\alpha = 10^{-1}, 10^{-2}, \ldots, 10^{-6}$. We
compute a reference solution $u_h$ (with $h=2^{-10}$) and the relative
error $\III{u_h - u_{H,k}}/\III{u_h}$ for each tuple
$(\IIH, \Omega^1, k, \alpha)$. The corresponding error bound for
$\IH$ can be found in Theorem~\ref{lem:full_error}.

\subsubsection{Stripes}
Here $A$ is defined as illustrated in
Figure~\ref{fig:num_k_stripes_coef}. The stripes are located at
distance $1/16$ from each other and their width is $1/128$. With
$H = 2^{-4}$ there is a node in each stripe and
Assumption~\ref{ass:localpoincare} is satisfied. We let
$f = \chi_{[1/4,3/4]^2}$ ($\chi$ is the indicator function) and use
homogeneous Dirichlet boundary conditions on the full boundary. The
relative errors for the the different interpolation operators and
contrasts are presented in
Figure~\ref{fig:num_k_stripes}. Quasi-monotonicity is not satisfied
for this mesh, but would require two additional mesh refinements to be
fulfilled.

\subsubsection{Random balls}
Here $A$ was constructed in the following manner. Starting with a
white background, for every node in the coarse mesh ($H=2^{-4}$),
randomly (chance 50\%) place a black ball of random radius between
$1/128$ and $8/128$ with center at the node. $A$ is defined according
to the convention of black representing $\Omega^1$ and white
representing $\Omega^\alpha$. See
Figure~\ref{fig:num_k_balls0_coef}. We let $f = \chi_{[1/4,3/4]^2}$
and use homogeneous Dirichlet boundary conditions on the full boundary.
The relative errors for the the different interpolation operators and
contrasts are presented in Figure~\ref{fig:num_k_balls0}. We note that
Assumption~\ref{ass:localpoincare} is satisfied since there is a node
in each ball. Quasi-monotonicity is not satisfied.

\subsubsection{Random field}
Here $A$ was constructed by a random process with spatial
correlation. See Figure~\ref{fig:num_k_random_coef}. We use
$f = \phi_{h,x}$ for $x = (1/2, 1/8)$ and impose homogeneous Neumann
boundary conditions on all edge segments except the right-most
($x_2 = 1$), where we impose homogeneous Dirichlet boundary
conditions. The relative errors for the the different interpolation
operators and contrasts are presented in
Figure~\ref{fig:num_k_random}. In this case, neither
Assumption~\ref{ass:localpoincare} nor quasi-monotonicity holds
for $H = 2^{-4}$.

\subsection{Discussion}
\label{sec:discussion}
We discuss the numerical results in the light of the theoretical
findings from previous sections. Although the localization error in
Lemma~\ref{lem:localization_error} is independent of the contrast, we
note that the contrast factor $\alpha^{-1}$ nevertheless enters the
full error bound in Theorem~\ref{lem:full_error}:
\begin{equation*}
  \III{u - u_{H,k}} \le C_{\Omega, \rho, \Cpoinc, \kappa, \Cdiam} \alpha^{-1} k^{d/2}\theta^{k} \|f\|_{L^2(\Omega)}.
\end{equation*}
The contrast factor $\alpha^{-1/2}$ also enters the following bound of
the energy norm of the true solution
\begin{equation*}
  \III{u}^2 \le \|u\|_{L^2(\Omega)} \cdot \|f\|_{L^2(\Omega)} \le C_{\Omega} \alpha^{-1/2} \|f\|_{L^2(\Omega)} \cdot \III{u}.
\end{equation*}
Thus, assuming the solution norm is proportional to this bound, the
relative error plotted in the figures behaves like
\begin{equation*}
  \III{u - u_{H,k}}/\III{u} \approx C_{\Omega, \rho, \Cpoinc, \kappa, \Cdiam} \alpha^{-1/2} k^{d/2}\theta^{k}.
\end{equation*}
We emphasize that $C_{\Omega, \rho, \Cpoinc, \kappa, \Cdiam}$ and
$0 < \theta < 1$ are proved to be independent of $\alpha$ for $\IH$
and $\IHproj$ (under different assumptions), and we can expect to
observe the contrast influencing the plotted relative errors by a
factor at most $\alpha^{-1/2}$. In particular, the decay rate of the
relative error in terms of patch size $k$ is independent of $\alpha$.

The following discussion is based on the results from the experiments,
shown in Figures~\ref{fig:num_k_stripes}--\ref{fig:num_k_random}. We
start by a few interpolation operator specific observations based on
all three experiments.
\begin{itemize}
\item
  The contrast problem is clearly visible for the full node patch
  Scott--Zhang interpolation operator $\IHsz$. It has a clear contrast
  dependent error for all three coefficients. The decay rate
  deteriorates as $\alpha \to 0$ in all cases.

\item
  We note that the novel interpolation operator $\IH$ shows
  contrast independent decay rates in all cases. It is interesting to
  note though, that $\IHinf$ (where $\delta=1$ instead of $1/4$) shows
  better decay rates for all coefficients. This suggests that there is
  room for improving the theoretical results to include this kind of
  operator (see Remark~\ref{rem:classIII}).

\item
  The error from using the nodal interpolation operator $\IHnodal$
  appears not to be contrast dependent, but the decay rate is low
  (except for the stripe coefficient, discussed below). The nodal
  interpolation operator is similar to $\IHdelta$ with
  $\delta = h/H = 2^{-6}$, which satisfies the presented theory on
  contrast independence, albeit with very low $L^2$-stability (large
  $\kappa$). This is probably what deteriorates the decay rate, but
  still keeps it contrast independent.
\end{itemize}
Next, we discuss the results on coefficient specific basis.
\begin{itemize}
\item For the stripes coefficient in Figure~\ref{fig:num_k_stripes},
  all operators except $\IHsz$ show good decay properties.  The nodal
  interpolation operator $\IHnodal$ shows an inverse relationship
  between error and contrast than what is generally expected. One
  interpretation is that when the contrast is high, this coefficient
  effectively constitute a number of weakly coupled 1D-problems. The
  nodal interpolation operator then leads to good localization
  according to the total element localization in 1D, discussed in
  e.g.\ \cite{HuSa07}. This also suggests that subdimensional channels
  (e.g.\ cracks, faults or fibers) can be handled accurately by this
  kind of node variables. The smoothing interpolation operator $\IHsz$
  seems to be particularly ill-suited for this kind of coefficient.

\item The random balls coefficient in Figure~\ref{fig:num_k_balls0}
  satisfies Assumption~\ref{ass:localpoincare}, however the
  quasi-monotonicity assumption does not hold. We can see that the
  decay rate of $\IHproj$ indeed depends on the contrast, while
  $\IHnodal$, $\IH$, $\IHinf$ and $\IHprojs$ enjoy contrast
  independent decay rates. It is interesting to note that the
  modification of $\IHproj$ to $\IHprojs$ (by selecting the
  integration domain to guarantee quasi-monotonicity) was sufficient
  to remove the contrast dependent decay. The operator $\IH$ shows
  contrast independent decay rate with a better rate than $\IHnodal$,
  but it appears to suffer from similar deterioration of decay rate as
  $\IHnodal$.

\item The random field coefficient in Figure~\ref{fig:num_k_random}
  satisfies neither Assumption~\ref{ass:localpoincare} nor
  quasi-mono\-tonicity. Despite this, the numerical results are
  similar to those for the random balls. Although the assumptions are
  not satisfied, the heavily contrast dependent error for $\IHsz$ in
  this experiment suggests that there can still be a gain in accuracy
  using the ideas presented here for coefficients not strictly covered
  by the assumptions.
\end{itemize}

\newcommand{\errorH}[2]{
  \begin{figure}[p]
    \centering
    \begin{subfigure}{0.48\textwidth}
      \centering
      \includegraphics[width=4cm, frame]{#1_coef.png}
      \caption{Coefficient}
      \label{fig:num_#1_coef}
    \end{subfigure}
    
    \begin{subfigure}{.48\textwidth}
      \centering
      \includegraphics[clip=true, trim=0cm 0.2cm 0cm 0cm]{plot_#1_SZfull_rhs}
      \caption{$\IHsz$}
    \end{subfigure}
    \begin{subfigure}{.48\textwidth}
      \centering
      \includegraphics[clip=true, trim=0cm 0.2cm 0cm 0cm]{plot_#1_nodal_rhs}
      \caption{$\IHnodal$}
    \end{subfigure}

    \begin{subfigure}{.48\textwidth}
      \centering
      \includegraphics[clip=true, trim=0cm 0.2cm 0cm 0cm]{plot_#1_SZpaper_rhs}
      \caption{$\IH$}
    \end{subfigure}%
    \begin{subfigure}{.48\textwidth}
      \centering
      \includegraphics[clip=true, trim=0cm 0.2cm 0cm 0cm]{plot_#1_SZsemi_rhs}
      \caption{$\IHinf$}
    \end{subfigure}

    \begin{subfigure}{.48\textwidth}
      \centering
      \includegraphics[clip=true, trim=0cm 0.2cm 0cm 0cm]{plot_#1_AProjectedClement_rhs}
      \caption{$\IHproj$}
    \end{subfigure}
    \begin{subfigure}{.48\textwidth}
      \centering
      \includegraphics[clip=true, trim=0cm 0.2cm 0cm 0cm]{plot_#1_AProjectedClementSpread_rhs}
      \caption{$\IHprojs$}
    \end{subfigure}%

    \begin{subfigure}{.96\textwidth}
      \centering
      \includegraphics[clip=true, trim=0cm 0.5cm 0cm 0cm]{plot_legend}
    \end{subfigure}
    \caption{Error versus $H$, $k=\log_2(H)-1$, #2.}
    \label{fig:num_#1}
  \end{figure}
}

\newcommand{\errork}[2]{
  \begin{figure}[p]
    \centering
    \begin{subfigure}{0.48\textwidth}
      \centering
      \includegraphics[width=4cm, frame]{#1_coef.png}
      \caption{Coefficient}
      \label{fig:num_k_#1_coef}
    \end{subfigure}
    
    \begin{subfigure}{.48\textwidth}
      \centering
      \includegraphics[clip=true, trim=0cm 0.2cm 0cm 0cm]{plot_k_#1_SZfull_rhs}
      \caption{$\IHsz$}
    \end{subfigure}
    \begin{subfigure}{.48\textwidth}
      \centering
      \includegraphics[clip=true, trim=0cm 0.2cm 0cm 0cm]{plot_k_#1_nodal_rhs}
      \caption{$\IHnodal$}
    \end{subfigure}

    \begin{subfigure}{.48\textwidth}
      \centering
      \includegraphics[clip=true, trim=0cm 0.2cm 0cm 0cm]{plot_k_#1_SZpaper_rhs}
      \caption{$\IH$}
    \end{subfigure}%
    \begin{subfigure}{.48\textwidth}
      \centering
      \includegraphics[clip=true, trim=0cm 0.2cm 0cm 0cm]{plot_k_#1_SZsemi_rhs}
      \caption{$\IHinf$}
    \end{subfigure}

    \begin{subfigure}{.48\textwidth}
      \centering
      \includegraphics[clip=true, trim=0cm 0.2cm 0cm 0cm]{plot_k_#1_AProjectedClement_rhs}
      \caption{$\IHproj$}
    \end{subfigure}
    \begin{subfigure}{.48\textwidth}
      \centering
      \includegraphics[clip=true, trim=0cm 0.2cm 0cm 0cm]{plot_k_#1_AProjectedClementSpread_rhs}
      \caption{$\IHprojs$}
    \end{subfigure}%

    \begin{subfigure}{.96\textwidth}
      \centering
      \includegraphics[clip=true, trim=0cm 0.5cm 0cm 0cm]{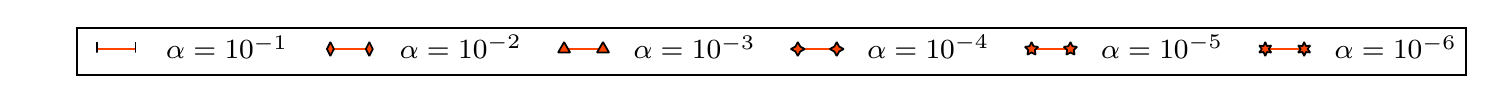}
    \end{subfigure}
    \caption{Error versus $k$, $H=2^{-4}$, $h=2^{-10}$, #2.}
    \label{fig:num_k_#1}
  \end{figure}
}

\errork{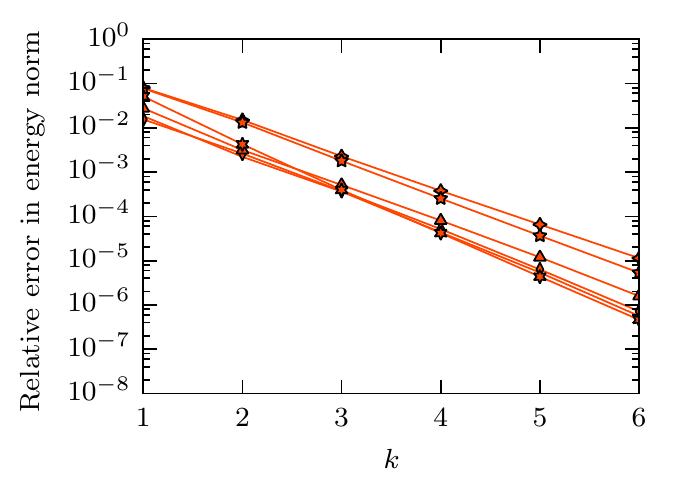}{stripes coefficient}
\errork{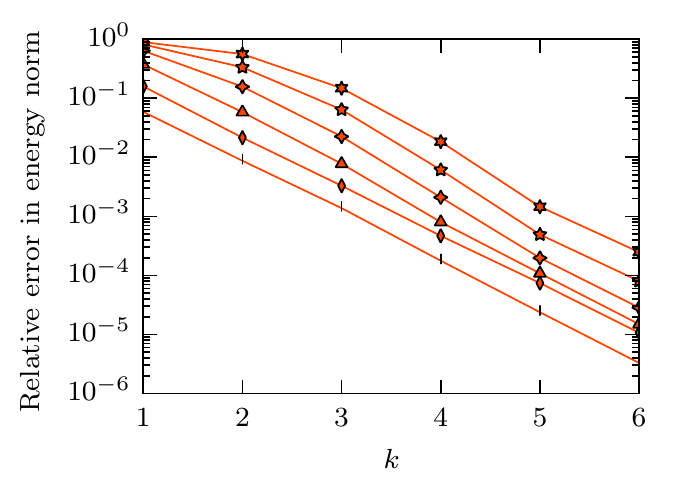}{random balls coefficient}
\errork{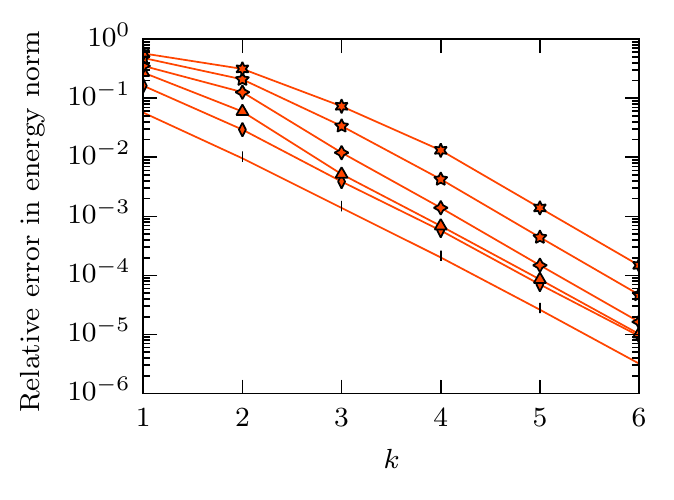}{random field coefficient}

\FloatBarrier

\bibliographystyle{abbrv} \bibliography{references}

\end{document}